\documentclass[9pt, english,english]{article}

\usepackage[applemac]{inputenc}

\usepackage{graphicx,epsfig,amsthm}
\usepackage{graphicx,epsfig,color}
\usepackage{amsmath,mathrsfs} 
\usepackage{amssymb} 
\usepackage{amsfonts}
\usepackage{mathtools}
\usepackage{epstopdf}
\usepackage{ifthen}
\usepackage{bbm}
\graphicspath{{eps/}}

\newcommand{\inR}{\in \mathbb{R}}
\newcommand{\inZ}{\in \mathbb{Z}}

\newcommand{\R}{ \mathbb{R}}
\newcommand{\Z}{ \mathbb{Z}}

\newcommand{\eqdef}{\stackrel{\vartriangle}{=}}

\newcommand{\Lop}{{\rm L}}
\newcommand{\Dop}{{\rm D}}

\def\V#1{{\boldsymbol{#1}}}         
\def\Spc#1{{\mathcal{#1}}}  
\def\M#1{{\bf{#1}}}  
\def\Op#1{{\mathrm{#1}}}  

\def\Proj{\mathrm{Proj}} 
\def\Identity{\mathrm{Id}} %

\renewcommand{\[}{\begin{equation}}
\renewcommand{\]}[1]{\label{eq:#1}\end{equation}}

\newtheorem{definition}{Definition}
\newtheorem{proposition}{Proposition}
\newtheorem{corollary}{Corollary}

\newtheorem{lemma}{Lemma}
\newtheorem{theorem}{Theorem}

\providecommand{\revise}[1]{{#1}}

\title{
Convex Optimization in Sums of Banach Spaces
\thanks{The research leading to these results has received funding from the Swiss National Science Foundation under Grant 200020\_184646 and from the European Research Council (H2020-ERC Project GlobalBioIm) under Grant 692726.
}
}

\author{
Michael Unser  and Shayan Aziznejad \thanks{Biomedical Imaging Group, \'Ecole polytechnique f\'ed\'erale de Lausanne (EPFL),
Station 17, CH-1015, Lausanne, Switzerland ({\tt michael.unser@epfl.ch}). }
 }

\begin{document}



\maketitle
\begin{abstract}
We characterize the solution of a broad class of convex optimization problems that address the reconstruction of a function from a finite number of linear measurements. 
The underlying hypothesis is that the solution is decomposable as a finite sum of components, where each component belongs to its own prescribed Banach space; moreover, the problem is regularized by penalizing some composite norm of the solution. We establish general conditions for existence and derive the generic parametric representation of the solution components. 
These representations fall into three categories depending on the underlying regularization norm: (i) a linear expansion in terms of predefined ``kernels'' when the component space is a reproducing kernel Hilbert space (RKHS),  
(ii) a non-linear (duality) mapping of a linear combination of measurement functionals when the component Banach space is strictly convex, and, (iii) an adaptive expansion in terms of a small number of atoms within a larger dictionary 
when the component Banach space  
is not strictly convex. Our approach generalizes and unifies a number of multi-kernel (RKHS) and sparse-dictionary learning techniques for compressed sensing available in the literature. It also yields the natural extension of the classical 
spline-fitting techniques in (semi-)RKHS to the abstract Banach-space setting.
\end{abstract}


\section{Introduction}
\subsection{From RKHS to Banach Spaces}
\label{Sec:RKHS1}
Reproducing kernel Hilbert spaces {(RKHS) play} a central role in the classical formulations of machine learning, statistical estimation, and the resolution of linear inverse problems \cite{Scholkopf2002,Berlinet2004}. They go hand in hand with quadratic (or Tikhonov) regularization and 
Gaussian processes \cite{Kimeldorf1970a, Badoual2018}. The popularity of RKHS in machine learning stems from the fact that the minimization of Hilbertian norms results in parametric solutions that are linear combinations of kernels (basis functions) centered on the data points \cite{Scholkopf2002,DEvito2004,Wendland2005, Hofmann2008}, a remarkable property that is supported by the celebrated representer theorem \cite{Scholkopf2001}.

However, recent works that revolve around the concept of sparsity have demonstrated the advantages of 
considering Banach spaces instead of Hilbert spaces. In particular, compressed sensing relies on the minimization of $\ell_1$-norms. Under suitable conditions, this enables the exact recovery of a signal from a limited number of linear measurements \cite{Donoho2006, Candes2007, Bruckstein2009, Elad2010b,Eldar2012, Foucart2013}. Researchers have established representer theorems that explain the sparsifying effect of the $\ell_1$-norm \cite{Unser2016} and of its variants, including its continuous-domain counterpart: the $\Spc M$-norm (a.k.a.\ the total-variation norm of a measure) \cite{Candes2014,Duval2015,Flinth2018,Bredies2018}. Likewise, we proved in  \cite{Unser2017} that non-uniform splines of a type that is matched to the regularization operator are universal solutions of linear inverse problems with generalized total-variation regularization.
The main difference with the RKHS (or Tikhonov) framework is that the underlying basis functions---or kernels---are selected in an adaptive fashion and are not necessarily placed on the data points \cite{Gupta2018}. More recently, we have shown that the effect of such minimum-norm regularization could be characterized in full generality, as described in Theorem \ref{Theo:GeneralRepBanach} below  \cite{Unser2020b,Boyer2018}. The latter is an ``abstract'' representer theorem that applies to any Banach space $\Spc X'$ (e.g., $\ell_\infty(\Z)=\big(\ell_1(\Z)\big)'$ or $\Spc M(\R^d)=\big(C_0(\R^d)\big)'$) identifiable as the dual of some primary Banach space $\Spc X$.

\subsection{From Sums of RKHS to Sums of Banach Spaces}
It is a known fundamental property that a convex combination (resp., a tensor product) of reproducing kernels retains the desirable reproducing-kernel property (a.k.a.\ positive-definiteness) \cite{Aronszajn1950}. This has prompted researchers to extend the single-kernel Hilbertian methods of machine learning to a whole range of composite problems that involve direct products or (internal) direct sums of RKHS. (It turns out that direct sums and direct products are topologically equivalent, which is the reason why direct-product spaces are sometimes referred to as {\em external direct sums} \cite{Megginson1998}.)
Examples of practical developments that involve direct product/sums of RKHS are: kernel methods for vector-valued data \cite{Micchelli2005,Alvarez2012}, multi-kernel learning \cite{Gonen2011,Micchelli2005b}, multiscale approximation \cite{VanWyk2000}, and 
semi-parametric models of the form $\widetilde{f}=f+p_0$, where $f \in \Spc H$ (RKHS) and the second component $p_0\in {\rm span}\{p_n\}_{n=1}^{N_0}$ is finite-dimensional \cite{Scholkopf2002}. 
Likewise, the native spaces of variational splines have an inherent direct-sum structure because the underlying regularization functional is a Hilbertian semi-norm \cite{deBoor1966,Duchon1977, Bezhaev2001, Mosamam2010}.

While the Banach counterparts of these methods are still lacking for the most part, there is recent evidence that the use of over-complete dictionaries---in particular, unions of bases---is highly advantageous for the resolution of compressed-sensing problems with sparsity constraints \cite{Elad2002generalized,Gribonval2003sparse,Starck2004redundant,Fuchs2004sparse,Rauhut2008compressed,Candes2011compressed,Lin2013compressed}.
In the case where the dictionary is a single basis, there is a direct relation between this type of signal recovery and the kind of $\ell_1$-regularization problem mentioned in Section \ref{Sec:RKHS1} \cite{Foucart2013}. 
By taking inspiration from the large body of work already available for RKHS, the next promising step is therefore  to investigate this type of reconstruction problem from the unifying perspective of an optimization in a sum of Banach spaces.

%


%

\subsection{Mathematical Context}
The aim of this paper is to characterize the solution of a broad class of unconstrained-optimization problems that address the recovery of some unknown function $f$ from a finite number of
(possibly noisy) samples or, more generally, from a set of $M$ linear measurements
$z_m=\langle \nu_m, f\rangle, m=1,\dots,M$.
Beyond the fact that we leave the classical framework of RKHS, the specificity of our investigation is that
the optimization is performed over some Banach space $\Spc X'$ that has a direct-sum structure and/or is equipped with a composite-norm.

To set the stage, we 
recall the primary results of \cite{Unser2020b} and \cite{Boyer2018} and introduce our abstract optimization framework in the form {of} a single unified theorem.

 \begin{theorem}[General Banach representer theorem]
\label{Theo:GeneralRepBanach}
Let us consider the following setting: 
\begin{itemize}
\item A dual pair $(\Spc X,\Spc X')$ of Banach spaces.
\item The analysis subspace $\Spc N_\V \nu={\rm span}\{\nu_m\}_{m=1}^M \subset \Spc X$
with the $\nu_m$ being linearly independent.
\item The linear measurement operator 
$\V \nu: \Spc X' \to \R^M: f \mapsto \big(\langle \nu_1,f \rangle, \dots,\langle \nu_M, f \rangle\big)$.
\item The  proper, lower-semicontinuous, 
and convex loss functional $E: \R^M \times \R^M \to \R^{+}\cup \{+\infty\}$.
\item Some arbitrary {strictly} increasing and convex function $\psi: \R^+ \to \R^+$. 
\end{itemize}
Then, for any fixed $\V y\inR^M$, the solution set of the generic optimization problem
\begin{align}
\label{Eq:GenericOptimizationProb}
S=\arg \min_{f \in \Spc X'} \left( E\big(\V y, \V \nu(f)\big)+ \psi\left( \|f\|_{\Spc X'}\right)\right)
\end{align}
is nonempty, convex, and weak$^\ast$-compact. 

When $E$ is strictly convex, or  if it imposes the  equality constraint
{$\V y = \V \nu (f)$}, then any solution $f_0 \in
S\subset \Spc X'$ is an $(\Spc X',\Spc X)$-conjugate of a common
$\nu_0 \in \Spc N_\V \nu\subset \Spc X$, so that $S\subseteq J(\nu_0)$ (see Definition \ref{Def:DualMap}). Depending on the type of Banach space, this then results in the following description of the solution(s):
\begin{itemize}
\item If $\Spc X'$ is a Hilbert space and $\psi$ is strictly convex, then
the solution is unique and admits the linear representation with parameter $\V a\inR^M$ given as
\begin{align}
\label{Eq:f0lin}
f_0=\sum_{m=1}^{M} a_m \varphi_m,
\end{align}
with $\varphi_m=\Op J_\Spc X\{\nu_m\}\in \Spc X'$, where
$\Op J_\Spc X$ is the Riesz map $\Spc X \to \Spc X'$.
\item If $\Spc X'$ is a 
strictly convex Banach space and $\psi$ is strictly convex, then
the solution is unique and admits the parametric representation
\begin{align}
\label{Eq:f0nonlin}
f_0=\Op J_{\Spc X}\left\{\sum_{m=1}^{M} a_m \nu_m\right\},
\end{align}
where
$\Op J_\Spc X$ is the (nonlinear) duality operator $\Spc X \to \Spc X'$ (see Definition \ref{Def:DualMap}).
\item  Otherwise, when $\Spc X'$ is not strictly convex,  the solution set is the weak*-closure of the convex hull of its extremal points (see Definition \ref{Def:Extremal}) which can all be expressed as
\begin{align}
\label{Eq:f0extreme}
f_0=\sum_{k=1}^{K_0} c_k e_k
\end{align}
for some $K_0\le M$, $c_1,\dots,c_{K_0} \inR$, where $e_1,\dots,e_{K_0}\in \Spc X'$ are some extremal points of the unit ball $B_{\Spc X'}=\{x \in \Spc X: \|x\|_{\Spc X'}\le 1\}$.
\end{itemize}
\end{theorem}
The definitions and mathematical background for the interpretation of Theorem \ref{Theo:GeneralRepBanach} are provided in Section \ref{Sec:Foundations}.
The first part of Theorem \ref{Theo:GeneralRepBanach} up to \eqref{Eq:f0nonlin} is a retranscription of
\cite[Theorem 1]{Unser2019c}. The representation \eqref{Eq:f0extreme} for the case where the solution is non-unique is then deducible from Theorem 3.1 of \cite{Boyer2018}.
Since the latter theorem 
is more general than what is required here, we are providing an alternative proof of the result in Appendix A. It is important to note that the linear expansion in \eqref{Eq:f0extreme}
 is adaptive, meaning that the actual choice of $K_0$ and of the basis functions $e_k\in \Spc X'$ is data-dependent.
This is the main difference with the two other cases for which Theorem \ref{Theo:GeneralRepBanach}
provides an explicit description of the $M$-dimensional solution manifold. 
\subsection{Contributions}
While the abstract characterization in Theorem \ref{Theo:GeneralRepBanach} is remarkably general, it is  practical only for the cases in which the duality operator
$\Op J_{\Spc X}: \Spc X \to \Spc X'$ or the extremal points of the unit ball in $\Spc X'$ are known explicitly, for instance when $\Spc X'$ is a RKHS \cite{Aronszajn1950,Wahba1990,Scholkopf2001} or when the underlying norm is a variant of the $\ell_1$-norm that promotes sparsity \cite{Unser2016,Boyer2018,Bredies2018}. In this paper, we are extending the applicability
of Theorem \ref{Theo:GeneralRepBanach} 
by starting from basic building blocks (elementary Banach constituents) and by showing how these can be combined via the use of linear transforms and of direct sums to specify more complex regularization norms that can accommodate mixture models.

In Section \ref{Sec:CompositeNorms}, we present a refinement of Theorem \ref{Theo:GeneralRepBanach} for the cases where $\Spc X'$ admits the decomposition $\Spc X'=\Spc X'_1 \times \dots \times \Spc X'_N$ (direct product of Banach spaces) or $\Spc X'=\Spc X'_1 \oplus \dots \oplus \Spc X'_N$ (direct sum of Banach spaces).
The main result there is Theorem \ref{Theo:Directprod}, which explicitly tells us how the
underlying direct-product duality mappings and extremal points can be determined from the knowledge of the same entities for the simpler constituent spaces $\Spc X'_n$.
In Section \ref{Sec:SumsBanach}, we focus our attention on the direct-sum scenario
and illustrate the relevance of our framework to the practice of signal processing and data science.
In particular, we present an alternative variational formulation for sparse-dictionary learning 
and 
a new representer theorem for mixed-norm regularization problems (Theorem \ref{Theo:L2M}).

In Section \ref{Sec:SemiNorm}, we extend Theorem \ref{Theo:GeneralRepBanach} by replacing the original regularizing norm by a semi-norm that has a finite-dimensional null space $\Spc N_\V p={{\rm span}\{p_1,\dots,p_{N_0}\}}$ $ \subset \Spc X'$. The main result expressed by Theorem \ref{Theo:GeneralRepSemiBanach} is that this adds a null-space
component $p_0$ to the generic solution(s) of Theorem \ref{Theo:GeneralRepBanach}, which is the desired outcome. 
At the same time, it reduces the intrinsic dimension of the complementary component
$s_0=f_0-p_0$
from $M$ to $(M-N_0)$.
The mathematical analysis amounts to making sure that the solution exists and to then properly split the problem in order to decouple the determination of the two solution components.
The significance of our new Theorem \ref{Theo:GeneralRepSemiBanach} is to show that the traditional techniques of 
spline approximation \cite{deBoor1966, Duchon1977, Wahba1990, Bezhaev2001}, which involve semi-reproducing-kernel Hilbert spaces \cite{Mosamam2010}, are extendable to Banach spaces in general. Likewise, the non-strictly convex scenario in Theorem \ref{Theo:GeneralRepSemiBanach} is consistent with a number of recent results that have appeared in the literature for sparsity-promoting functionals \cite{Unser2017,Boyer2018,Bredies2018}, although the overlap is only partial due to the generality of our present formulation.

\section{Mathematical Foundations}
\label{Sec:Foundations}
A Banach space is a complete normed vector space. It is denoted by $(\Spc X,\|\cdot\|_{\Spc X})$ where $\Spc X$ stands for the vector space and $\|\cdot\|_{\Spc X}$ specifies the underlying norm
or, simply, by $\Spc X$ (for short). A Banach space $\Spc X$ has a unique topological dual
$\Spc X'$ which is itself a Banach space equipped with the dual norm $\|\cdot\|_{\Spc X'}$ (see \eqref{Eq:DualNorm} below).
Formally, an element $f$ of the dual space $\Spc X'$ is a continuous linear functional $f: \Spc X \to \R$. 
Likewise, since $\Spc X$ is embedded in the bidual space $\Spc X''=(\Spc X')'$, an element $\nu \in \Spc X$, which is therefore also included in $\Spc X''$, can be viewed as a continuous linear functional $\nu: \Spc X' \to \R$.
The bilateral character of this association is described by the duality product
\begin{align}
\Spc X \times \Spc X' \to \R: (\nu,f) \mapsto \langle \nu, f\rangle_{\Spc X \times \Spc X'}=\langle f,\nu\rangle_{\Spc X' \times \Spc X}\in \R,
\end{align}
which is a map that is linear and continuous in both arguments. 
To avoid notational overload, we shall henceforth drop the subscript in the specification of the duality product under the understanding that the first argument is a linear functional that acts on the second argument; for instance,  $\nu: f \mapsto \langle \nu, f\rangle$, where $f \in \Spc X'$ usually also has a concrete identification as a vector or a function.
Mathematically, the continuity of the duality product or, equivalently, the continuity of $\nu$ (or of $f$) viewed as a linear functional---is expressed by the generic duality bound
\begin{align}
\label{Eq:DualityBound}
\big|\langle \nu, f\rangle\big| 
\le \|\nu\|_{\Spc X} \|f\|_{\Spc X'} ,
\end{align}
which holds for any $(\nu,f) \in \Spc X \times \Spc X'$---for more details, refer to  \cite{Megginson1998,Rudin1991}. 
The
upper bound in \eqref{Eq:DualityBound} is consistent with the definition of the 
dual norm
\begin{align}
\label{Eq:DualNorm}
\|f\|_{\Spc X'}\eqdef\sup_{\nu\in\Spc X\backslash\{0\}} 
\frac{\langle f,\nu\rangle}{\|\nu\|_{\Spc X}}.
\end{align}
In fact, the latter identification suggests that
the
bound in \eqref{Eq:DualityBound} is tight---a 
property that is embodied in the fundamental notion of duality mapping \cite{Beurling1962}.
\begin{definition}[Duality mapping]
\label{Def:DualMap}
Let $(\Spc X,\Spc X')$ be a dual pair of Banach spaces.
Then, the elements $\nu^\ast\in \Spc X'$ and $\nu \in \Spc X$ form a $(\Spc X',\Spc X)$-conjugate pair
if they satisfy: \begin{enumerate}
\item Norm preservation: $\|\nu^\ast\|_{\Spc X'}=\|\nu\|_{\Spc X}$.
\item Sharp duality bound: $\langle \nu,\nu^\ast\rangle
=\|\nu\|_{\Spc X}\; \|\nu^\ast\|_{\Spc X'}.$
\end{enumerate}
For any given $\nu \in \Spc X$, the set of admissible conjugates defines the duality mapping 
\begin{align}
J(\nu)=\{\nu^\ast \in \Spc X': \|\nu^\ast\|_{\Spc X'}=\|\nu \|_{\Spc X}  \mbox{ and }  \langle \nu, \nu^\ast \rangle
=\|\nu\|_{\Spc X}\; \|\nu^\ast\|_{\Spc X'}\},
\end{align}
which is a nonempty subset of $\Spc X'$.
Whenever the duality mapping is single-valued (for instance, when $\Spc X'$ is strictly convex), one 
also defines the duality operator
$\Op J_\Spc X: \Spc X \to \Spc X'$, which is such that
$\nu^\ast=\Op J_\Spc X\{\nu \}$.
 \end{definition}
 \begin{definition} 
\label{Def:StrictConvex}
A Banach space $\Spc X$ (or its associated norm $\|\cdot\|_{\Spc X}$) is said to be 
{\it strictly convex} if,
for all $f_1,f_2 \in \Spc X$ such that $\|f_1\|_{\Spc X}=\|f_2\|_{\Spc X}=1$ and $f_1\ne f_2$, one has that ${\|\lambda f_1+(1-\lambda)f_2\|_{\Spc X}}<1$
for any $\lambda\in(0,1)$.
\end{definition}

The dual mapping is a powerful mathematical tool that facilitates the investigation of optimization problems in Banach spaces. A primary reference on the topic, which includes the characterization of $\Op J_\Spc X$ for the classical $L_p$ spaces, is \cite{Cioranescu2012}.

Note that the duality operator $\Op J_\Spc X$ is bijective
when $\Spc X$ is reflexive and strictly convex \cite{Cioranescu2012,Schuster2012}, in which case
$\Op J^{-1}_\Spc X=\Op J_{\Spc X'}: \Spc X' \to \Spc X''=\Spc X$.
It can therefore be viewed as the natural generalization of the celebrated Riesz map \cite{Riesz1927, Rudin1973}, which
describes the linear isometric mapping of a Hilbert space into its dual.
The important difference, however, is that the operator $\Op J_\Spc X$ is generally nonlinear.
In fact, it is linear if and only if $\Spc X$ is a Hilbert space,
in which case it coincides with the Riesz map $\Spc X \to \Spc X'$ \cite{Cioranescu2012,Unser2020b}. This explains the
distinction between the first and second scenarios in Theorem \ref{Theo:GeneralRepBanach}:
The simplification in \eqref{Eq:f0lin}
 occurs because we are able to move the operator inside the sum of \eqref{Eq:f0nonlin}. 
 
The final statement in Theorem \ref{Theo:GeneralRepBanach}, which applies to the cases where the solution is non-unique, involves the notion of extremal points.
  \begin{definition}[Extremal points]
 \label{Def:Extremal}
Let $C$ be a convex set of a Banach space $\Spc X$.
The extremal points of $C$ are the points $f\in C$ such that, if there exist $f_1,f_2 \in C$
and $t\in(0,1)$ such that
$f=t f_1 + (1-t) f_2$, then it necessarily holds that $f=f_1=f_2$. The set of these extremal points is denoted by ${\rm Ext}(C)$.
\end{definition}
While the characterization given by \eqref{Eq:f0extreme} is always valid, 
it is practical only when the Banach space
 $\Spc X'$ has a unit ball $B_{\Spc X'}$ with comparatively much fewer extremal points than boundary points.
 The prototypical case is $\ell_1(\Z)$, whose extremal points ${\rm Ext}(B_{\ell_1(\Z)})=\{\pm\delta[\cdot-m]\}_{m\inZ}$ (the signed Kronecker impulses shifted by $m$) are indexable, while its boundary points
$\{u[\cdot] \in \ell_1(\Z): \|u\|_{\ell_1}=1\}$ are uncountable.
The extremal points of $B_{\Spc X'}$ 
can then be interpreted as the elements of a constrained dictionary.
This means that $f_0$ in \eqref{Eq:f0extreme} will be constructed by adaptively 
selecting a few elements $e_k$ (with $K_0 \ll M$ when the solution is strongly regularized)
in a dictionary of preferred elementary solutions. This is a very popular approach in compressed sensing \cite{Donoho2006, Candes2007, Elad2010b,Chandrasekaran2012}.

Our final tool is a 
transformation mechanism that generates appli\-cation-specific Banach spaces
from some primary ones whose basic properties (e.g., the duality mapping and extremal points of the unit ball) are known.

 \begin{proposition}[Isometric isomorphism]
\label{Prop:IsometricIso} Let $\Spc X$ be a primary Banach space and $\Op T: \Spc X \to \Op T(\Spc X)$ a linear operator that is injective on $\Spc X$.
Then, we have the following properties:
\begin{enumerate}
\item The space $\Spc Y=\Op T(\Spc X)=\{y=\Op T x: x \in \Spc X\}$, equipped with the norm $\|y\|_{\Spc Y}=\|\Op T^{-1}\{y\}\|_{\Spc X}$, is a Banach space that is isometrically isomorphic to
$\Spc X$. In other words, the operators
$\Op T: \Spc X \to \Spc Y$ and $\Op T^{-1}: \Spc Y \to \Spc X$ are isometries.

\item The continuous dual of $\Spc Y=\Op T(\Spc X)$ is $\Spc Y'=\Op T^{-1\ast}(\Spc X')$, equipped with the norm 
{$\|y^\ast\|_{\Spc Y'}=\|\Op T^\ast  \{y^\ast\} \|_{\Spc X'}$}.

\item The elements $y^\ast\in \Spc Y'$ and $y\in \Spc Y=\Op T(\Spc X)$ 
form a conjugate pair if and only if
{ $x^\ast = \Op T^{\ast} \{y^\ast\}\in \Spc X'$ and $x = \Op T^{-1} \{y\} \in \Spc X$} are themselves $(\Spc X',\Spc X)$-Banach conjugates.

\item The element $u\in \Spc Y$ is an extremal point of the unit ball in $\Spc Y=\Op T(\Spc X)$
if and only if $u=\Op T\{e\}$, where $e\in \Spc X$ is an extremal point of the unit ball in $\Spc X$.
\item If $\Spc X$ is a Hilbert space, then the spaces 
$\Spc Y=\Op T\big(\Spc X\big)$ and 
{$\Spc Y'=\Op T^{-1\ast}\big(\Spc X'\big)$} are Hilbert spaces as well. The corresponding Riesz map is 
 $\Op J_\Spc Y=\Op T^{-1\ast}\Op J_\Spc X\Op T^{-1}:\Spc Y \to \Spc Y'$, where
  $\Op J_\Spc X: \Spc X\to \Spc X'=\Op J_\Spc X(\Spc X)$ is the Riesz map of the primary space.
\end{enumerate}
\end{proposition}
\begin{proof}
The hypothesis that $\Op T$ is injective on $\Spc X$ implies the existence of a linear map $\Op T^{-1}$ (inverse operator) such that
$\Op T^{-1}\Op T \{x\}=x$ for all $x \in \Spc X$. Since $\Op T$ is linear and one-to-one, 
the functional $y \mapsto \|\Op T^{-1}y\|_{\Spc X}$ is a {\it bona fide} norm
on $\Spc Y$. Moreover,  
from the definition of the $\Spc Y$-norm, we have that
\begin{equation}\label{Eq:Isometry}
\|\Op T\{x_m \} -\Op T\{x_n\} \|_{\Spc Y}=\|\Op T\{x_m -x_n\}\|_{\Spc Y}=\|\Op T^{-1}\Op T\{x_m-x_n\}\|_{\Spc X}  = \|x_m-x_n\|_{\Spc X},
\end{equation}
for any $x_m,x_n \in \Spc X$. Together with the bijectivity of $\Op T$, we deduce that $\Op T$ is an isomorphism between $\Spc X$ and $\Spc Y$. Hence, $\Spc Y$ inherits the topological structure of $\Spc X$. This proves that $\Spc Y$ is indeed a Banach space.

The other properties are immediate consequences of the underlying isometry and the definition of the adjoint, which translate into
\begin{align*}
\langle x_1^\ast, x_2\rangle_{\Spc X' \times \Spc X}=\langle x_1^\ast,\Op T^{-1} \Op T \{x_2\}\rangle_{\Spc X' \times \Spc X}=\langle \Op T^{-1\ast}\{x_1^\ast\}, \Op T \{x_2\}\rangle_{\Spc Y' \times \Spc Y}=
\langle y_1^\ast, y_2\rangle_{\Spc Y' \times \Spc Y}
\end{align*}
for any $(x_1^\ast,x_2) \in \Spc X' \times \Spc X$.

In particular, if $\Spc X$ is a Hilbert space with inner product
$(\cdot,\cdot)_\Spc X$, then $x^\ast=\Op J_\Spc X \{x\} \in \Spc X'$ so that
$\langle x^\ast, x\rangle_{\Spc X' \times \Spc X}=\|x\|^2_{\Spc X}=(x,x)_{\Spc X}$.
It follows that
$\Spc Y=\Op T\big( \Spc X\big)$ is a Hilbert space equipped with the inner product $(y_1,y_2)_{\Spc Y}
=(\Op T^{-1}y_1,\Op T^{-1}y_2)_{\Spc X}$. Correspondingly, the dual space $\Spc Y'=\Op T^{-1\ast}\big( \Spc X'\big)$ is the Hilbert space equipped with the inner product
$(y^\ast_1,y^\ast_2)_{\Spc Y'}
=(\Op T^{\ast}y^\ast_1,\Op T^{\ast}y^\ast_2)_{\Spc X'}$. Moreover, we have that
$(y_1,y_2)_{\Spc Y}=\langle \Op J_\Spc Y\{y_1\},y_2\rangle_{\Spc Y' \times \Spc Y}=(\Op J_\Spc Y\{y_1\},\Op J_\Spc Y\{y_2\})_{\Spc Y'}$,
the underlying duality operator (Riesz map) being
$\Op J_\Spc Y=\Op T^{-1\ast}\Op J_\Spc X\Op T^{-1}: \Spc Y \to \Spc Y'$. 
\end{proof}

\section{Composite Norms and Direct-Sum Spaces}
\label{Sec:CompositeNorms}
In order to offer flexibility in the specification of direct-product or direct-sum topologies, we introduce the finite-dimensional space $\Spc Z=(\R^N,\|\cdot\|_\Spc Z)$.
The underlying norm is said to be {\em monotone} if
$$
\|(a_1, \dots,a_N)\|_{\Spc Z} \le \|(b_1, \dots,b_N)\|_{\Spc Z} 
$$
whenever $0 \le |a_n|\le  |b_n|$  for each $n=1,\dots,N$, and, {\em absolute} if 
$\|\V z\|_{\Spc Z}=\|(z_n)\|_{\Spc Z}=\|(|z_n|)\|_{\Spc Z}$ for any $\V z\inR^N$. It is also known that a norm is monotone if and only if it is absolute  \cite[Theorem 2]{Bauer1961}. 
For instance, the latter property is obviously satisfied for $\|\cdot\|_\Spc Z=\|\cdot\|_p$ with $p\ge1$, as well as for any weighted version thereof.
Moreover, the dual of an absolute norm is again absolute \cite[Theorem 1]{Bauer1961}.
Given a series $\Spc X_1, \dots, \Spc X_N$ of Banach spaces,
we then write $(\Spc X_1\times\dots \times \Spc X_N)_{\Spc Z}$ for the direct-product space
equipped with the composite norm
\begin{align}
\label{Eq:CompNorm1}
\|(x_1,\dots,x_N)\|=\|(\|x_1\|_{\Spc X_1}, \dots, \|x_N\|_{\Spc X_N})\|_\Spc Z.
\end{align}
The construction is straightforward as the direct-product space automatically inherits the Banach property of its components. 

Likewise, one can construct (internal) direct-sum spaces via the summation of complemented Banach constituents.
\begin{definition}
A series $\Spc X_1, \dots, \Spc X_N$ of Banach subspaces of $\Spc X$ is said to be complemented if
$\Spc X=\Spc X_1 + \dots + \Spc X_N= \{x=x_1+ \dots+ x_N: x_n \in \Spc X_n, n=1,\dots,N\}$
(as a set) and
$\Spc X_{n_1} \cap \sum_{n\ne n_1}\Spc X_{n}=\{0\}$ when $n_1=1,\dots,N$.
\end{definition}
In that scenario, any $x \in \Spc X$ has a unique representation as
$x=x_1 + \dots + x_N$ with $x_n=\Proj_{\Spc X_n}\{x\} \in \Spc X_n$,
where $\Proj_{\Spc X_n}: \Spc X \to \Spc X_n$ is the corresponding projection operator.
We then designate $\Spc X=(\Spc X_1\oplus\dots \oplus \Spc X_N)_{\Spc Z}$ as the (internal) direct-sum space
equipped with the norm
\begin{align}
\label{Eq:CompNorm2}
\|x\|_{\Spc X}=\|(\|\Proj_{\Spc X_1}\{x\}\|_{\Spc X_1}, \dots, \|\Proj_{\Spc X_N}\{x\}\|_{\Spc X_N})\|_{\Spc Z}.
\end{align}
We observe that \eqref{Eq:CompNorm2} is compatible with \eqref{Eq:CompNorm1} 
because $\Proj_{\Spc X_{n_1}}: \Spc X \to \Spc X_{n_1}$ is such that
$$
\Proj_{\Spc X_{n_1}}\{x_{n}\}=\left\{ \begin{array}{ll}
  x_{n_1}, & \mbox{for } n=n_1    \\
 0, &  \mbox{ otherwise}   \\
 \end{array}
\right. 
$$
for any $x_{n} \in \Spc X_n$.
This identification, together with the unicity of the sum decomposition, implies that $(\Spc X_1\oplus\dots \oplus \Spc X_N)_{\Spc Z}$ is a Banach space that 
 is isometrically isomorphic to $(\Spc X_1\times\dots \times \Spc X_N)_{\Spc Z}$.

\begin{lemma}
\label{Theo:DirectSumOpt} Let $(\Spc X'_1,\Spc X_1),\dots,(\Spc X'_N,\Spc X_N)$ be a series
of dual pairs of Banach spaces and $\|\cdot\|_{\Spc Z}$ a norm on $\R^N$ that is absolute.
Then, we have the following properties:
\begin{enumerate}
\item The continuous dual of $\Spc X=(\Spc X_1\times\dots \times \Spc X_N)_{\Spc Z}$ is the 
direct-product space $\Spc X'=(\Spc X'_1\times\dots \times \Spc X'_N)_{\Spc Z'}$.

\item The elements $y=(y_1,\dots,y_N)\in \Spc X'$ and $x=(x_1,\dots,x_N)\in \Spc X$ 
form a conjugate pair if and only if
$y_n=\alpha_n x_n^\ast$, where $x^\ast_n\in \Spc X_n'$ is a Banach conjugate of $x_n\in \Spc X_n$
and $\alpha_n\inR^+$ is given by
\begin{align}
\label{Eq:DirectSumConj}
\alpha_n=\left\{
\begin{array}{ll}
 \frac{z^\ast_n}{\|x_n\|_{\Spc X_n}}>0, & x_n\ne 0    \\
 0, &   \mbox{otherwise}   \\
\end{array}\right.
\end{align}
with $\V z^\ast=(z^\ast_n)
\in \Spc Z'$ a Banach conjugate of
$\V z=(\|x_1\|_{\Spc X_1},\dots,\|x_N\|_{\Spc X_N}) \in \Spc Z$.

\item The element $e=(e_1,\dots,e_N)\in \Spc X$ is an extremal point of the unit ball in $\Spc X$
if and only if $(\|e_1\|_{\Spc X_1},\dots,\|e_N\|_{\Spc X_N})$ is an extremal point of the unit ball in $\Spc Z$, and for each $1\le n \le N$ with $e_n\ne0$, $\frac{e_n}{\|e_n\|_{\Spc X_n}}$ is an extremal point of the unit ball of $\Spc X_n$.
\item If the $\Spc X_n$ are complemented Banach subspaces of the (sum) space $\Spc X_{\rm sum}$, then
the continuous dual of $\Spc X_{\rm sum}=(\Spc X_1\oplus\dots \oplus \Spc X_N)_{\Spc Z}$ is the 
direct-sum Banach space $\Spc X'_{\rm sum}=(\Spc X'_1\oplus\dots \oplus \Spc X'_N)_{\Spc Z'}$, which is isometrically 
isomorphic to the direct-product space $\Spc X'$ in Item 1. Consequently, the properties in Item 2 and 3 also apply, with the convention that $x_n=\Proj_{\Spc X_n}\{x\}$
and $y_n=\Proj_{\Spc X'_n}\{y\}$ for $n=1,\dots,N$.
\end{enumerate}

\end{lemma}

\begin{proof}
An element $y=(y_1,\dots,y_N)$ of $\Spc X'$ is identified with the linear functional
\begin{align}
\label{Eq:multifunctional}
x=(x_1,\dots,x_N) \mapsto \langle y, x\rangle_{\Spc X' \times \Spc X}=
\sum_{n=1}^N \langle y_n, x_n\rangle_{\Spc X'_n \times \Spc X_n}.
\end{align}
The first property is a basic result in the theory of Banach spaces \cite[Theorem 1.10.13]{Megginson1998} when
the outer norm is Euclidean with
$\Spc Z=\Spc Z'=(\R^N,\|\cdot\|_2)$.
The present setting is more general so that we need to  prove that the dual norm of $y=(y_1,\ldots,y_N) \in \Spc X'$ is precisely 
\begin{equation}\label{Eq:DualNorm2}
\|y\|_{\Spc X'}= \sup_{\|x\|_{\Spc X} =1} \langle y,x \rangle  = 
\revise{
\left\| (\|y_1\|_{\Spc X'_1},\ldots,\|y_N\|_{\Spc X'_N})\right\|_{\Spc Z'}
}.
\end{equation}
Since the spaces $(\Spc X'_n,\Spc X_n)$ form dual pairs, we have the generic duality inequalities
\begin{align}
\label{Eq:innerdualitybound0}
\langle y_n, x_n\rangle_{\Spc X'_n \times \Spc X_n}\le \left| \langle y_n, x_n\rangle_{\Spc X'_n \times \Spc X_n}\right|\le \|y_n\|_{\Spc X'_n}\|x_n\|_{\Spc X_n}
\end{align}
with equality if and only if $y_n=\alpha_n x^\ast_n$ for some $\alpha_n \inR^+$. 
This implies that, for any $(y,x) \in \Spc X' \times \Spc X$, we have that
\begin{align}
\label{Eq:innerdualitybound}
\langle y, x\rangle_{\Spc X' \times \Spc X}=\sum_{n=1}^N \langle y_n, x_n\rangle_{\Spc X'_n \times \Spc X_n}\le \sum_{n=1}^N \big|\langle y_n, x_n\rangle_{\Spc X'_n \times \Spc X_n}\big| \le \sum_{n=1}^N \|y_n\|_{\Spc X'_n}\|x_n\|_{\Spc X_n}
\end{align}
Likewise, by setting $\V y=(\|y_1\|_{\Spc X'_1},\dots,\|y_N\|_{\Spc X'_N}) \in \Spc Z'$ and
$\V z=(\|x_1\|_{\Spc X_1},\dots,\|x_N\|_{\Spc X_N})\in \Spc Z$, we 
write the complementary duality inequality
\begin{align}
\label{Eq:outerdualitybound}
\sum_{n=1}^N \|y_n\|_{\Spc X'_n}\|x_n\|_{\Spc X_n}
=\langle \V y, \V z \rangle_{\Spc Z' \times \Spc Z}\le 
\left|\langle \V y, \V z \rangle_{\Spc Z' \times \Spc Z}\right| \le \|\V y\|_{\Spc Z'} \|\V z\|_{\Spc Z}.
\end{align}
By observing that
$\|\V z\|_{\Spc Z}=\|x\|_{\Spc X}$ and combining these inequalities, we get that
\begin{align}
\langle y, x\rangle_{\Spc X' \times \Spc X}\le \sum_{n=1}^N \|y_n\|_{\Spc X'_n}\|x_n\|_{\Spc X_n}\le 
  \|\V y\|_{\Spc Z'} \|\V x\|_{\Spc Z}=
\|\V y\|_{\Spc Z'}\|x\|_{\Spc X},
\end{align}
which shows that $\|y\|_{\Spc X'}$ is upper-bounded by $\|\V y\|_{\Spc Z'}=\left\| (\|y_1\|_{\Spc X'_1},\ldots,\|y_N\|_{\Spc X'_N})\right\|_{\Spc Z'}$.
To prove that we actually have $\|y\|_{\Spc X'} =\|\V y\|_{\Spc Z'}$, for any $\epsilon>0$,   we need to find $x_{\epsilon}\in \Spc X$ with $\|x_{\epsilon}\|_{\Spc X}=1$ such that  
$$
  \langle y, x_{\epsilon}\rangle_{\Spc X' \times \Spc X}\geq \|\V y\|_{\Spc Z'}-\epsilon.
$$
By definition of the dual norm $\|\cdot\|_{\Spc Z'}$, we have that
\begin{equation}\label{Eq:DualZNorm}
\|\V y\|_{\Spc Z'}= \sup_{\substack{\V \alpha \in \mathbb{R}^N \\ \revise{\|\V \alpha\|_{\Spc Z}\le1}}} \V y^T \V\alpha.
\end{equation}
Since $\mathbb{R}^N$ is a finite-dimensional vector-space, the unit ball $B_{\Spc Z} = \{\alpha \in \mathbb{R}^N: \|\V \alpha\|_{\Spc Z}\le 1\}$ is compact. Hence, there exists a vector $\V \alpha =(\alpha_1,\ldots,\alpha_N)\in B_{\Spc Z}$ that attains the supremum in \eqref{Eq:DualZNorm}. In other words, 
$$
\|\V y\|_{\Spc Z'}= \V y^T \V \alpha=\sum_{n=1}^N \|y_n\|_{\Spc X'_n} \alpha_n.
$$
Similarly, for any $\epsilon>0$, the definition of the dual norm implies the existence of unit-norm elements $x_n \in \Spc X_n$ for $n=1,\ldots,N$ such that 
\begin{equation}\label{Eq:EpsilonApprox}
 \langle y_n, x_n\rangle_{\Spc X'_n \times \Spc X_n}\geq \|y_n\|_{\Spc X'_n}-\frac{2\epsilon}{N(\alpha_n^2+1)}.
\end{equation}
We then set $x_{\epsilon} =(\alpha_1 x_1, \ldots, \alpha_N x_N)\in\Spc X$ and observe that
$$
\|x_{\epsilon}\|_{\Spc X} = \left\|(\|\alpha_1 x_1\|_{\Spc X_1},\ldots,\|\alpha_N x_N\|_{\Spc X_N}) \right\|_{\Spc Z}= \|(|\alpha_1|,\ldots,|\alpha_N|)\|_{\Spc Z} = 1.
$$
Based on \eqref{Eq:EpsilonApprox} and the inequality $\frac{\alpha}{\alpha^2+1} \leq \frac{1}{2}$ for all $\alpha \in \R$, we then deduce that 
\begin{align*}
 \langle y, x_{\epsilon}\rangle_{\Spc X' \times \Spc X} & = \sum_{n=1}^N \langle y_n, \alpha_n x_n\rangle_{\Spc X'_n \times \Spc X_n} \geq \sum_{n=1}^N \alpha_n\left(\|y_n\|_{\Spc X'_n}-\frac{2\epsilon}{N(\alpha_n^2+1)}\right)\\&= \sum_{n=1}^N \alpha_n\|y_n\|_{\Spc X'_n}  - \frac{2\epsilon}{N} \sum_{n=1}^N\frac{\alpha_n}{\alpha_n^2+1} = \|\V y\|_{\Spc Z'}  - \frac{2\epsilon}{N} \sum_{N=1}^N\frac{\alpha_n}{\alpha_n^2+1} \geq \|\V y\|_{\Spc Z'} - \epsilon,
\end{align*}
which, in light of the inequality $\|y\|_{\Spc X'} \leq \|\V y\|_{\Spc Z'}$, allows us to conclude that $\|y\|_{\Spc X'} = \|\V y\|_{\Spc Z'}$.

To prove the second property, we observe that $y\in \Spc X'$ and $x \in \Spc X$ form a conjugate pair if and only if an equality occurs in both \eqref{Eq:innerdualitybound} and \eqref{Eq:outerdualitybound}. Inequalities \eqref{Eq:innerdualitybound0} and \eqref{Eq:innerdualitybound} are saturated if and only if $y_n=\alpha_n x^\ast_n$, $\alpha_n\in\R^+$, and $(x_n^\ast,x_n)$ form a $(\Spc X'_n,\Spc X_n)$-conjugate pair. 
The saturation of \eqref{Eq:outerdualitybound} with $\|y\|_{\Spc X'}=\|\V y\|_{\Spc Z'}=\|\V z\|_{\Spc Z}=\|x\|_{\Spc X}$ is then equivalent to 
$\V y=\V z^\ast=(z^\ast_1,\dots,z^\ast_N)$. Under the assumption that $x_n\ne0$, this yields $\alpha_n=\frac{z^\ast_n}{\|x_n^\ast\|_{\Spc X'_n}}$, which is the announced result since 
$\|x_n^\ast\|_{\Spc X'_n}=\|x_n\|_{\Spc X_n}$. 

The third property is due to Dowling and Saejung  \cite{Dowling2008} under the assumption that the $\|\cdot\|_{\Spc Z}$-norm is absolute and monotone in the positive orthant; in other words, when the condition $0\leq a_n \leq b_n$ for $n=1,\ldots,N$ implies that $\|\V a\|_{Z} \leq \|\V b\|_{Z}$. 
By invoking Bauer's theorem \cite{Bauer1961}, we are able to drop the (redundant) assumption of monotonicity since it is implied by the absoluteness property.

The last statement is a direct consequence of the isometric isomorphism between
$\Spc X_{\rm sum}=(\Spc X_1 \oplus \dots, \oplus \Spc X_N)_{\Spc Z}$ and
$\Spc X=(\Spc X_1 \times \dots, \times \Spc X_N)_{\Spc Z}$.
\end{proof}

In particular, if $\|\cdot\|_{\Spc Z}=\|\cdot\|_2$ is the usual Euclidean norm, then
$\V z^\ast$ in Property 2 is unique and coincides with $\V z$, which implies that
the Banach conjugate of $\V x=(x_1,\dots,x_N) \in \Spc X$ is simply 
$\V x^\ast=(x^\ast_1, \dots, x^\ast_N)\in\Spc X'$.

The combination of these preparatory results and Theorem \ref{Theo:GeneralRepBanach} allows us to deduce the following.

\begin{theorem}[Representer theorem for direct-product spaces]
\label{Theo:Directprod}
If the space $\Spc X'$ in Theorem 1 has a direct-product decomposition as
$\Spc X'=(\Spc X'_1\times\dots \times \Spc X'_N)_{\Spc Z'}$ with predual
 $\Spc X=(\Spc X_1\times\dots \times \Spc X_N)_{\Spc Z}$, where 
$(\Spc X'_1,\Spc X_1),\dots,(\Spc X'_N,\Spc X_N)$ are dual pairs of Banach spaces and both $E$ and $\psi$ are strictly convex, then
the solutions $f_0=(f_{0,1},\dots,f_{0,N})\in S\subset \Spc X'$ of the optimization problem \eqref{Eq:GenericOptimizationProb}
are $(\Spc X',\Spc X$)-Banach conjugates of a common $${\nu}_0=({\nu}_{0,1},\dots,{\nu}_{0,N})=\sum_{m=1}^M a_m \nu_m,$$
where
$\nu_m=(\nu_{m,1},\dots,\nu_{m,N})\in \Spc X$ 
with $\nu_{m,n} \in \Spc X_n$
and a suitable set of coefficients $\V a\inR^M$. 



Moreover, 
%
%
%
depending of the properties of the underlying Banach constituents, the
solution components $f_{0,n} \in \Spc X'_n$ have the following characterization 
with predefined scaling constants 
\begin{align}
\label{Eq:alphas}
\alpha_n=\left\{
\begin{array}{ll}
 \frac{y_n}{y^\ast_n}>0, & y_n\ne 0    \\
 0, &   \mbox{otherwise},   \\
\end{array}\right.
\end{align}
where $\V y=(\|f_{0,1}\|_{\Spc X_1'},\dots,\|f_{0,N}\|_{\Spc X_N'})$ and 
$\V y^*=\Op J_{\Spc Z'}\{\V y\}$:
\begin{itemize}
\item If $\Spc X_n'$ is a Hilbert space and 
$\Spc Z'$ is strictly convex, then
$f_{0,n}$ is unique and admits the linear representation
\begin{align}
\label{Eq:f0MultiHilbert}
f_{0,n}=\alpha_n\sum_{m=1}^{M} a_m \varphi_{m,n}
\end{align}
with $\varphi_{m,n}=\Op J_{\Spc X_n}\{\nu_{m,n}\}\in \Spc X_n'$, where
$\Op J_{\Spc X_n}$ is the Riesz map $\Spc X_n \to \Spc X_n'$.
\item If $\Spc X_n'$ is a 
strictly convex Banach space and 
$\Spc Z'$ is strictly convex, then
the solution component is unique and admits the parametric representation
\begin{align}
\label{Eq:f0MultiBanach}
f_{0,n}=\alpha_n\Op J_{\Spc X_n}\left\{\sum_{m=1}^{M} a_m \nu_{m,n}\right\}
\end{align}
where
$\Op J_{\Spc X_n}$ is the (nonlinear) duality operator $\Spc X_n \to \Spc X_n'$.
\item If $\Spc X_n'$ is a non-strictly convex Banach space, then the subcomponent solution set 
$S|_{\Spc X_n'}$ is the weak*-closure of the convex hull of its extremal points, which can all be expressed as
\begin{align}
\label{Eq:f0MultiSparse}
f_{0,n}= \sum_{k=1}^{K_{0}} c_{k,n} e_{k,n},
\end{align}
where 
$e_{1,n},\dots,e_{K_0,n}\in \Spc X_n'$ are some extremal points of the unit ball in ${\Spc X_n'}$
and $c_{1,n},\dots,c_{K_0,n} \inR$ some appropriate weights;
the (minimal) number of atoms
$K_{0}\le M$ is
common to all the components associated with non-reflexive Banach spaces.
\end{itemize}

In the particular case where $\|\cdot\|_{\Spc Z'}=\|\cdot\|_{1}$, \eqref{Eq:f0MultiSparse} can be replaced by
\begin{align}
\label{Eq:MultiSparse}
f_{0,n}=\sum_{k=1}^{K_n} c_{k,n} e_{k,n}
\end{align}
with $\sum_{n=1}^N K_n\le M$.
In addition, \eqref{Eq:f0MultiBanach} (resp. \eqref{Eq:f0MultiHilbert}) remains valid for the components for which the space $\Spc X_n'$ is 
strictly convex (resp., Hilbertian), with the caveat
that the solution is no longer guaranteed to be unique; this, then, contributes a degenerate version of
\eqref{Eq:MultiSparse} with $K_{n}=1$, $c_{1,n}=\|f_{0,n}\|_{\Spc X'_n}$, and 
$e_{1,n}=f_{0,n}/\|f_{0,n}\|_{\Spc X'_n}$.
\end{theorem}
\begin{proof}
The existence of solutions $f_0 \in \Spc X'$ and the property that $S \subseteq J_\Spc X(\nu_0)$ for some $\nu_0=\sum_{m=1}^M a_m \nu_m \in \Spc N_\V \nu$ is ensured by Theorem 1.
We then proceed in three steps.\\[-2ex]

\item (i) Constant value of $\psi(\|f_0\|_{\Spc X'})$ for all $f_0\in S$.\\
The key here is the {strict convexity} of $f \mapsto E(\V y,\V \nu(f))$ 
together with the convexity of $f\mapsto \psi(\|f\|_{\Spc X'})$. By applying a standard argument  (by contradiction) that uses the convexity of $S$, we 
show that there exist two constants $C_1$ and $C_2$ such that
$E(\V y,\V \nu(f_0))=C_1$ and $\psi(\|f_0\|_{\Spc X'})=C_2$ for all
$f_0 \in S$ (see, for instance, the last part of the proof in \cite[Appendix B]{Gupta2018}). By invoking the strict convexity of $E$, this then implies that all solutions share the same measurement vector
$\V z_0=\V \nu(f_0)$. Likewise, when $\psi$ is strictly convex, we readily deduce that $\|f_0\|_{\Spc X'}$ takes a constant value.\\[-2ex]


\item (ii) \revise{Uniqueness} of $\|f_{0,n}\|_{\Spc X_n'}$ in the strictly-convex case.\\
To show that $\|f_{0,n}\|_{\Spc X_n'}=y_n$ holds for all $f_0\in S$, we suppose that
there exists another solution $\widetilde f_0\in S$ such that $\|\widetilde f_0\|_{\Spc X'}=\|f_0\|_{\Spc X'}$
and $\|\widetilde f_{0,n}\|_{\Spc X_n'}=\widetilde y_n$ with $\widetilde{\V y}\ne \V y$.
Since $S$ is convex, $\lambda \widetilde f_0 + (1-\lambda) f_0$ with any $\lambda\in(0,1)$ must also be a solution
with associated norm $\|\lambda \widetilde f_0+ (1-\lambda) f_0\|_{\Spc X'}\le  \|\lambda \widetilde{\V y}+ (1-\lambda) \V y\|_{\Spc Z'}$, by the triangle inequality.
However, the norm equality $\| \widetilde f_0\|_{\Spc X'}=\|\widetilde{\V y}\|_{\Spc Z'}= \|\V y\|_{\Spc Z'}$ and the strict-convexity of $\|\cdot\|_{\Spc Z'}$ (see Definition \ref{Def:StrictConvex}) implies that
 $ \|\lambda \widetilde{\V y} + (1-\lambda) \V y\|_{\Spc Z'}< \|\V y\|_{\Spc Z'}=\| f_0\|_{\Spc X'}$, which results in a contradiction.\\[-2ex]

\item (iii) Generic form of the solution component $f_{0,n}$.\\
We assume that
$y_n=\|f_{0,n}\|_{\Spc X'}\ne0$; otherwise, we simply have that $f_{0,n}=0$.
From Property 2 in Lemma \ref{Theo:DirectSumOpt}, we know that $f_0=(f_{0,1},\dots,f_{0,N})$
and $\nu_0=(\nu_{0,1},\dots,\nu_{0,N})$ form a conjugate pair if and only if there exists
$f^\ast_{0,n}\in J_{\Spc X_n'}(f_{0,n})$ such that
$\nu_{0,n}=(y_n^\ast/y_n) f^\ast_{0,n}$, where $\V y^\ast=(y_1^\ast,\dots,y^\ast_N)=\Op J_{\Spc X'}\{\V y\}$. 

When $\Spc X_n'$ is strictly convex,
the duality mapping is
single-valued. 
The representations in \eqref{Eq:f0MultiHilbert} and \eqref{Eq:f0MultiSparse} then directly follow from the primary expansion $\nu_{0,n}=\sum_{m=1}^M a_m \nu_{m,n}$ and the homogeneity property of the duality mapping expressed as $\Op J_{\Spc X}\{\alpha \nu\}=\alpha \Op J_{\Spc X}\{\nu\}$ for any $\nu \in \Spc X$ and $\alpha\inR^+$ (see \cite{Cioranescu2012}).

Since $S$ is convex and weak$^\ast$-compact, we can invoke the Krein-Milman theorem,
which states that $S$ is the 
closure of the convex hull of its extremal points.
The same holds true for the convex set $S|_{\Spc X'_n}$ (the restriction of $S$ on $\Spc X_n'$)
with  ${\rm Ext}(S|_{\Spc X'_n})  \subseteq {\rm Ext}(S)|_{\Spc X'_n}$.
By recalling that all points $f_{0}\in {\rm Ext}(S)$ can be represented as
$f_{0}=(f_{0,1},\dots,f_{0,N})=\sum_{k=1}^{K_0} c_k e_k$,  where $e_k=(e_{k,1},\dots,e_{k,N})\in{\rm Ext}(B_{\Spc X'})$ and $K_0\le M$ (by Theorem \ref{Theo:GeneralRepBanach}), we obtain that
\begin{align}
\label{Eq:f0n}
f_{0,n}=\sum_{k=1}^{K_0}  c_k \|e_{k,n}\|_{\Spc X'_n}\widetilde e_{k,n},
\end{align}
where  $\widetilde e_{k,n}=e_{k,n}/\|e_{k,n}\|_{\Spc X'_n}$ are extremal points of the unit ball
in $\Spc X_n'$ (by Lemma \ref{Theo:DirectSumOpt}, Property 3). %
The announced statement with $c_{k,n}=\|e_{k,n}\|_{\Spc X'_n} c_k$
then follows from the property that \eqref{Eq:f0n}
 is valid for all
$f_{0,n} \in {\rm Ext}(S)|_{\Spc X'_n}\supseteq {\rm Ext}(S|_{\Spc X'_n})$.
In fact, Property 3 in Lemma \ref{Theo:DirectSumOpt} tells us that the subset of points $f_{0,n} \in {\rm Ext}(S|_{\Spc X'_n})$
are those for which $\V e_k=\V y/\|\V y\|_{\Spc Z'}$ are extremal points of the unit ball in $\Spc Z'$.
In particular, when $\|\cdot\|_{\Spc Z'}=\|\cdot\|_{1}$ (outer $\ell_1$-norm), the $\V e_k$ all have the binary form $(0,0,\dots,\pm 1,0,\dots)$ with a single active coefficient at $n=n_k$, which then yields \eqref{Eq:MultiSparse}.
 \end{proof}
The outcome of Theorem \ref{Theo:Directprod} is that the generic form of the solution in Theorem \ref{Theo:GeneralRepBanach} is essentially transferred to the direct-product components, with the distribution of the relative
energy being controlled by 
the outer norm $\|\cdot\|_{\Spc Z'}$. The effect of the $\ell_1$-norm is significant in that respect because it acts as a threshold that selectively blocks certain solution components and lets others through.

We wish to highlight the fact that the arguments for the proof of Theorem \ref{Theo:Directprod} (as well as Theorem 1)
\revise{involves} neither a calculus of variations nor a recourse to the sophisticated machinery of Fr\'echet derivatives and subgradients.
It only requires the Hahn-Banach theorem and the characterization of the configurations that
saturate the underlying duality 
inequalities.
\section{Convex Optimization in Sums of Banach Spaces}
\label{Sec:SumsBanach}
The techniques that we 
describe next are relevant to inverse problems for 
which the solution $f_0$ can be decomposed into a sum of components that have distinct smoothness and/or sparsity properties. The solution then lives in a sum of Banach spaces. Beside the reconstruction of $f_0$ from the noisy measurement $\V y=\V \nu(f)+\V \epsilon$, we are now faced with the additional challenge of disambiguating the individual components of the solution.

Let $\Spc X'_1, \dots,\Spc X'_N$ be a series of Banach spaces whose elements are indexed over the same domain. We then define the sum space
$$
\Spc X'_1 + \dots + \Spc X'_N=\{f=f_1+f_2+\dots+f_N: f_n \in \Spc X_n', n=1,\dots,N\}.
$$
Given a linear measurement operator
$\revise{\V \nu=(\nu_1,\dots,\nu_M)}: \Spc X'_1 + \dots + \Spc X'_N \to \R^M$ with
$\nu_n \in \cap_{n=1}^N \Spc X_n$ and a set of measurements $\V y \inR^M$, we are then interested in the study of the solvability of the
convex optimization problem
\begin{align}
\label{Eq:MultiOptimizationProb}
S=\arg \min_{(f_n)_{n=1}^N: f_n\in \Spc X'_n}  
\left(E\Big(\V y, \V \nu(\sum_{n=1}^N f_n)\Big)+ \psi\left(\left\| (\|f_1\|_{\Spc X'_1},\dots,\|f_N\|_{\Spc X'_N})\right\|_{\Spc Z'} \right)\right)
\end{align}
where the functions $E$ and $\psi$ are the same as in Theorem \ref{Theo:GeneralRepBanach}, while $\|\cdot\|_{\Spc Z'}$ is a suitable norm that controls the coupling of the components.
The idea here is to segregate the components $f_n$ by
favouring some ``regularized'' solutions $f_{0}=(f_{0,n})_{n=1}^N$ such that the $\|f_{0,n}\|_{\Spc X'_n}$ are small in an appropriate sense.
Problem \eqref{Eq:MultiOptimizationProb} is generally well defined. 
Its solution can be obtained as a special case of Theorem \ref{Theo:Directprod}.
To see this, it suffices to invoke the linearity of $\nu_m$, which yields
$$
\nu_m\Big(\sum_{n=1}^N f_n\Big)=\sum_{n=1}^N \langle \nu_m, f_n\rangle_{\Spc X_n \times \Spc X'_n}=\langle \widetilde{\nu}_m, f\rangle_{\Spc X \times \Spc X'}
$$
with $f=(f_1,\dots,f_N) \in \Spc X'=(\Spc X'_1 \times \dots \times \Spc X'_N)_{\Spc Z'}$, and $\widetilde{\nu}_m=(\nu_m,\dots,\nu_m) \in \Spc X=(\Spc X_1 \times \dots \times \Spc X_N)_{\Spc Z}$.
The multicomponent optimization problem \eqref{Eq:MultiOptimizationProb} is therefore equivalent to \eqref{Eq:GenericOptimizationProb}
with $\Spc X'$ being a direct-product space and the specific choice of a ``replicated'' measurement operator $\widetilde{ \V \nu}=(\V \nu,\dots, \V \nu)$.
Consequently, we get the general form of the solution
by simple substitution of $\nu_{m,n}$ by $\nu_m$ in Theorem \ref{Theo:Directprod}.
We shall now illustrate the power of the approach by considering special cases that are motivated by applications.

\subsection{Multicomponent Learning in RKHS/RKBS}
\label{Sec:multiRKHS}
In the classical supervised learning (or regression) setting \cite{Bishop2006}, one is given a series of data points
$(\V x_m,y_m)\inR^{d}\times \R$, $m=1,\dots,M$. The goal
is to determine a function $f: \R^d \to \R$ such that
$f(\V x_m)\approx y_m$ for all $m$ without overfitting the data, which is the reason why one generally imposes some regularization on the solution.

We make the link with our framework by considering the ``sampling'' functionals
$\V \nu =(\delta(\cdot-\V x_1),\dots,\delta(\cdot-\V x_M))$ with $\delta(\cdot-\V x_m): f \mapsto f(\V x_m)$ (Dirac impulse shifted by $\V x_m$), which is such that
$\V \nu(f)=(f(\V x_1),\dots,f(\V x_M))$.
To enable the sequential handling of data, one
quantifies the goodness of fit with some additive loss functional of the form
$\sum^M_{m=1} E\big(y_m, f(\V x_m)\big)$ with  $E: \R \times \R \to \R^{+}\cup \{+\infty\}$,
the simplest case being the least-squares criterion with $E(y_m,f(\V x_m))=|y_m-f(\V x_m)|^2$.
The traditional form of regularization is the squared Hilbertian norm $\psi(\|f\|_{\Spc X'})=\|f\|^2_{\Spc H}$ with $\psi(\cdot)=|\cdot|^2$ (strictly convex) in the
reproducing-kernel Hilbert space $\Spc X'=\Spc H$.
\begin{definition} [See \cite{Aronszajn1950}]
A Hilbert space $\Spc H$ of functions on $\R^d$ is called a reproducing kernel Hilbert space (RKHS) if its dual
$\Spc H'$ is such that
$\delta(\cdot-\V x) \in \Spc H'$ for any $\V x \inR^d$.
Then, the unique representer $r_\Spc H(\cdot,\V x)=\Op J_{\Spc H'}\{\delta(\cdot-\V x)\}=\delta^\ast(\cdot-\V x) \in \Spc H$ when  indexed by $\V x$ is called the reproducing kernel of the Hilbert space.
\end{definition}
The ``reproducing'' qualifier refers to the basic property that
$$(r_\Spc H(\cdot,\V x),f)_{\Spc H}=(\delta^\ast(\cdot-\V x),f)_\Spc H=\langle \delta(\cdot-\V x),f\rangle_{\Spc H' \times \Spc H}=f(\V x)$$
for all $f\in \Spc H$ and any $\V x\inR^d$.
We now use Theorem  \ref{Theo:Directprod} to obtain a multi-kernel extension of Sch\"olkopf's celebrated
representer theorem for learning in RKHS \cite{Scholkopf2001}.
The classical theorem corresponds to the scenario of a single RKHS with $N=1$, for which the underlying kernel is simply $r(\cdot,\V y)=r_{\Spc H}(\cdot,\V y)=\delta^\ast(\cdot-\V y)$.

\begin{corollary}[Multi-kernel expansion in RKHS]
\label{Theo:MultiKernel}
Let us consider the following setting:
\begin{itemize}
\item A series of reproducing-kernel Hilbert spaces $\Spc H_1, \dots, \Spc H_N$ whose members are functions on $\R^d$. The reproducing kernel of $\Spc H_n$ is  
$r_{\Spc H_n}: \R^d \times \R^d \to \R$.
\item A strictly convex loss functional $E: \R \times \R \to \R^{+}.$ 
\item A strictly increasing and convex function $\psi: \R^+ \to \R^+$.
\item An absolute norm $\|\cdot\|_{\Spc Z}$  on $\R^N$.

\end{itemize}
Then,  for any given series of points $(\V x_m,y_m)\inR^{d+1}$, $m=1,\dots,M$, the multi-component data-fitting problem
\begin{align}
\label{Eq:MultiOptimizationProb2}
S=\arg \min_{f =f_1 + \cdots+f_N: f_n\in \Spc H_n}  \left(\sum_{m=1}^M
E\big(y_m, f(\V x_m)\big)+ \psi\big(\big\|(\|f_1\|_{\Spc H_1},\dots,\|f_N\|_{\Spc H_N})\big\|_{\Spc Z'}\big)\right)
\end{align}
always admits a global solution of the form
\begin{align}
\label{Eq:SolmultiRKHS}
f_0(\V x)=\sum_{m=1}^M a_m r(\V x,\V x_m)
\end{align}
with $(a_1,\dots,a_M) \inR^M$, where the underlying multi-kernel is given by
\begin{align}
\label{Eq:SolmultiRKHS2}
r(\V x,\V y)=\sum_{n=1}^{N} \alpha_n r_{\Spc H_n}(\V x,\V y),
\end{align}
with suitable weights $(\alpha_1, \cdots,\alpha_N) \in \R_{\ge 0}^N$. Moreover, the solution is unique if $\psi(\cdot)$ and $\|\cdot\|_{\Spc Z'}$ are both strictly convex.
\end{corollary}

\begin{proof}
Problem \eqref{Eq:MultiOptimizationProb2} is a special case of \eqref{Eq:MultiOptimizationProb} with $\Spc X'_n=\Spc H_n=\Spc H_n''$
and $\nu_m=\delta(\cdot-\V x_m) \in \Spc X_n=\Spc H'_n$, due to the RKHS property.
We then apply
\eqref{Eq:f0MultiHilbert} in Theorem \ref{Theo:Directprod}
with $\nu_{m,n}=\delta(\cdot-\V x_m)$, which gives the parametric form of the solution components with $\varphi_{m,n}=r_{\Spc H_n}(\cdot,\V x_m)=\Op J_{\Spc H'_n}\{\delta(\cdot-\V x_m)\}$
where $r_{\Spc H_n}$ is the reproducing kernel of $\Spc H_n$.
 The summation constraint $f_0=f_{0,1}+\dots+f_{0,N}$ and the property that the coefficients $\V a\in \R^M$ are shared by all components then 
gives \eqref{Eq:SolmultiRKHS}. 
\end{proof}

The interesting aspect in Corollary \ref {Theo:MultiKernel} is that the underlying kernel given by \eqref{Eq:SolmultiRKHS2}
is tunable, which offers flexibility and is in the line with certain forms of multiple-kernel learning \cite{Gonen2011}.

The simplest choice of regularization functional in \eqref{Eq:MultiOptimizationProb2} is the weighted sum
$\sum_{n=1}^N \lambda_n \|f_n\|^2_{\Spc H_n}$, where the
$\lambda_n>0$ are appropriate regularization parameters. This corresponds to  $\psi(\cdot)=|\cdot|^2$ in \eqref{Eq:MultiOptimizationProb2} with the outer weighted Euclidean norm 
$$
\|\V y\|_{\Spc Z'}=\|\V y\|_{2, \V \lambda}=\left( \lambda_1 y_1^2 + \cdots + \lambda_N y_N^2 \right)^{\frac{1}{2}},
$$
whose dual is 
$\|\cdot \|_{2, \V \mu}$ with $\V \mu=(1/\lambda_1,\dots,1/\lambda_N)$.
Correspondingly, the ($\Spc Z$-$\Spc Z')$ conjugate of $\V y\in \Spc Z'$ is
$$
\revise{\V y^\ast=(\lambda_1 y_1,\dots,\lambda_N y_N)} \in \Spc Z
$$
which, with the help of \eqref{Eq:alphas} in Theorem \ref{Theo:Directprod}, then yields $\alpha_n=y_n/y^\ast_n=1/\lambda_n$ and offers direct control over \eqref{Eq:SolmultiRKHS2}.

By selecting the outer norm $\|\cdot\|_{\Spc Z'}$ to be non-Euclidean, one can also make the kernel-shaping effect in \eqref{Eq:SolmultiRKHS2} data-dependent, with the effect becoming more pronounced as we relax the ``strictness'' of the convexity requirement. 
The prototypical case of a regularization functional that falls into the latter category is $\lambda \sum_{n=1}^N  \|f_n\|_{\Spc H_n}$ (mixed $\ell_1$-norm) with parameter $\lambda>0$,
which favors sparse kernel configurations. Micchelli and Pontil \cite{Micchelli2005b} have shown that the corresponding minimization problem is equivalent to a special instance of kernel learning
where the ``optimal'' kernel $r(\cdot,\V x)$ is selected within the multi-kernel dictionary
$$
\Spc K=\left\{\sum_{n=1}^N \alpha_n r_{\Spc H_n}(\cdot,\V x): \alpha_n\ge 0, \alpha_1+\cdots+\alpha_N=1\right\},
$$
which is consistent with the last statement in Theorem \ref{Theo:Directprod}.

The approach is generalizable to reflexive Banach spaces with the caveat that the resulting representer model is no longer linear.
\begin{definition}[See \cite{Zhang2009,Zhang2012b}]
A strictly convex and reflexive Banach space $\Spc B$ of functions on $\R^d$ is called a reproducing-kernel Banach space (RKBS) if
$\delta(\cdot-\V x) \in \Spc B'$ for any $\V x \inR^d$.
Then, the unique representer $r_\Spc B(\cdot,\V x)=\Op J_{\Spc B'}\{\delta(\cdot-\V x)\}=\delta^\ast(\cdot-\V x) \in \Spc B$ when  indexed by $\V x$ is called the reproducing kernel of the Banach space.
\end{definition}

We can then consider the direct Banach counterpart of the multicomponent problem in Corollary \ref{Theo:MultiKernel}
with $f_n \in \Spc B_n=\Spc X'_n$, $\delta(\cdot-\V x_m) \in \Spc B'_n=\Spc X''_n=\Spc X_n$ and derive the general parametric form of the solution as
\begin{align}
\label{Eq:SolmultiRKBS}
f_0=f_1+\cdots+f_N \mbox{ with } f_n=\alpha_n\Op J_{\Spc B_n'}\left\{ \sum_{m=1}^M a_m \delta(\cdot-\V x_m)\right\}
\end{align}
with $(a_m)\inR^M$ and $(\alpha_n)\inR_{\ge0}^N$.
The result provided by \eqref{Eq:SolmultiRKBS} is new to the best of our knowledge. It is the multi-component extension of  \cite[Theorem 2]{Zhang2012b}, which is itself a reformulation and slight generalization of \cite[Theorem 19]{Zhang2009}.
\subsection{Sparse Signal Representation in Dictionaries}
\label{Sec:SparseDico}
While the use of classical Hilbertian smoothness (e.g., Sobolev) norms lends itself to a closed-form resolution of \eqref{Eq:MultiOptimizationProb}, the underlying multicomponent model takes its full power when the regularization norms are not strictly convex and promote sparsity. This statement is supported by a large body of work in compressed sensing (CS).
To show how this fits the present formulation, we now consider the prototypical CS problem:
the recovery of a vector $\M x\inR^N$ from its linear measurements
$\M H \M x=\M y\inR^M$,
where
$\M H$ 
is the system matrix with $M$ much smaller than $N$.
Given a series of invertible matrices $\M L_i \in \R^{N \times N}$, we then specify
the component Banach spaces $\Spc X'_i=(\R^N,{\|\cdot\|_{\Spc X_i'}})$
with $\|\M x\|_{\Spc X_i'}=\|\M L_i\M x\|_1$. There, the use of the $\ell_1$-norm is intended to promote sparsity in the transformed domain associated with $\M L_i$.

In order to be able to apply Theorem \ref{Theo:Directprod}, we need to identify the
predual spaces $\Spc X_i$ as well as the extremal points
of the unit ball in $\Spc X_i'$. To that end, we invoke Proposition \ref{Prop:IsometricIso}
with the primary pair of dual spaces $\Spc X=(\R^N,\|\cdot\|_{\infty})$, $\Spc X'=(\R^N,\|\cdot\|_{1})$ and with $\Op T_i: \M x \mapsto \M L^T_i\M x$.
This then results in the identification of the predual Banach space
$\Spc X_i=\Op T_i(\Spc X)=(\R^N,\|\cdot\|_{\Spc X_i})$ with 
$\|\M x\|_{\Spc X_i}=\|\M L_i^{-T}\M x\|_\infty$
and its dual  
$\Spc X'_i = (\mathbb{R}^N,\|\cdot\|_{\Spc X'_i})$ with
\begin{align}
\label{eq:Xisparse}
\|\M y \|_{\Spc X'_i} = \sup_{\|\M x\|_{\Spc X_i}=1} \M y^T \M x = \sup_{\|\M L_i^{-T}\M x\|_\infty=1} \M y^T \M x 
= \sup_{\|\M v\|_\infty=1} \M y^T \M L_i^T \M v = \| \M L_i \M y\|_1,
\end{align}
which is consistent with the definition $\|\M x\|_{\Spc X'_i}=\|\M L_i \M x\|_1=\|\M c_i\|_1$.  

It is well known that the extremal points of the unit ball in $(\R^N,\|\cdot\|_1)$
are $\pm\M e_n$ where $\M e_n$  is the $n$th element of the canonical basis with $[\M e_n]_m=\delta_{m-n}$.
We then apply the isometric isomophism $\Spc Y_i'=\Op T_i^{-1\ast}(\Spc X')$ with
$\Op T_i^{-1\ast}: \M e \mapsto \M L_i^{-1} \M e$
(see Properties 2 and 4 in Proposition \ref{Prop:IsometricIso})
to deduce that ${\rm Ext}(B_{\Spc X_i'})=\{\pm \M u_{n,i}\}_{n=1}^N$, where 
$B_{\Spc X_i'}$ is the unit ball of the Banach space $(\mathbb{R}^N,\|\cdot\|_{\Spc X_i'})$ and 
\begin{align}
\label{Eq:l1extrem}
\M u_{n,i}=\M L_i^{-1} \M e_n.
\end{align}
In fact, the latter defines a (sub)-dictionary whose elements are the row vectors of $\M L_i^{-1}$.

We now have the tools to characterize the solution(s) of the multi-component $\ell_1$ optimization problem
\begin{align}
\label{Eq:DictionaryProb}
S=\arg \min_{\M x =\M x_1 + \dots + \M x_I \inR^N}  \left(
\|\M y- \M H \M x\|_2^2+    \lambda \sum_{i=1}^I \|\M L_i \M x_i \|_1\right),
\end{align}
which is a special case of \eqref{Eq:MultiOptimizationProb} with
$\nu_m: \M x \mapsto \M h_n^T \M x$,
$\psi=\Identity$, and $\|\cdot \|_{\Spc Z'}=\|\cdot\|_1$
where $\M h_n^T$ is the $nth$ row vector of $\M H$.
First, we confirm that the problem is well posed by observing that
$\M h \in \Spc X_i$ for any measurement vector $\M h \inR^N$.
In other words, \eqref{Eq:DictionaryProb} always admits a solution, albeit not necessarily a unique one.
To obtain the 
parametric form of the extremal solution components $(\M x_{0,1},\dots,\M x_{0,M})$, we apply the last part of Theorem \ref{Theo:Directprod}, which yields
$$
\M x_{0,i}=\sum_{k=1}^{K_i} c_{k,i} \M u_{n_k, i}
$$
with $\sum_{i}^I K_i\le M$ and $\sum_{i=1}^I \|\M L_i \M x_{0,i}\|_1=\sum_{i=1}^I\|\M c_i\|_1$.
This may also be rewritten as
\begin{align}
\label{Eq:dicosolution}
\M x_0 =\sum_{i=1}^I \M x_{0,i}=\sum_{k=1}^{K_0} c_{k} \M u_{n_k, i_k}
\end{align}
with $K_0\le M$ and some appropriate weights $(c_k)\inR^{K_0}$ and
$\sum_{i=1}^I\|\M x_{i,0}\|_{\Spc X_i'}=\|\M c\|_1$.
In effect, \eqref{Eq:dicosolution} tells us that the solution is constructed by picking $K_0$ atoms (with $K_0\le M$) in an enlarged dictionary 
$$
\M U=[\M u_{1,1} \  \cdots \ \M u_{N,1} \ \M u_{1,2} 
\ \cdots \M u_{2,N} \ \cdots \cdots \ 
\M u_{I,N} ] \inR^{N \times (I \times N)},$$ 
which is formed from the unions of the $\M u_{n,i}$ in \eqref{Eq:l1extrem} with $i=1,\dots,I$, 
$n=1,\dots,N$.
The result in \eqref{Eq:dicosolution} also motivates us to reformulate Problem \eqref{Eq:DictionaryProb} in a more familiar ``synthesis'' form
\begin{align}
\label{Eq:DictionaryProb2}
S=\arg \min_{\M c \inR^{N\times I }}  \left(
\|\M y- \M H \M U \M c\|^2_2 +  \lambda \| \M c\|_1\right).
\end{align}
where $\M c\inR^{N \times I}$ is an augmented parameter vector.
The latter is the standard LASSO formulation for the recovery of a signal subject to the constraint that it has a sparse representation in some predefined dictionary $\M U$. The new elements here are the link with the ``analysis'' form \eqref{Eq:DictionaryProb} and the guarantee of the existence of a ``sparse'' solution with $K_0\le M$, irrespective of whether the conditions for uniqueness (e.g., restricted isometry) of CS are met or not.

\subsection{Signal Recovery Problems Involving Mixed Norms}
While the examples of Section \ref{Sec:multiRKHS} and \ref{Sec:SparseDico} 
are fairly classical, we can use  our high-level results
to derive some new representer theorems, such as the following example
which involves a combination of smoothness (RKHS) and sparsity-promoting regularizations.

As prerequisite, we need to specify dual pairs of Banach spaces that are matched to specific norms and regularization operators $\Lop_i$.
To that end, we assume that the linear operator $\Op T=\Lop^\ast$ is injective on some primary space $\Spc X$ 
and recall the relevant results from Proposition \ref{Prop:IsometricIso}:
\begin{enumerate}
\item The operator $\Lop^\ast$ is invertible on its range $\Op L^\ast\big(\Spc X\big)$.
\item The dual pair of Banach spaces $\Spc B=\Op L^\ast\big(\Spc X\big)$ and 
$\Spc B'=\Op L^{-1}\big(\Spc X'\big)$ is isometrically isomorphic to $(\Spc X,\Spc X')$.
\item The operator $\Lop$ isometrically maps $\Spc B'=\Op L^{-1}\big(\Spc X'\big)$ to $\Spc X'$.
\item If $\Spc X=L_2(\R^d)$, then $\Spc H=\Op L^\ast\big(L_2(\R^d)\big)$ and
$\Spc H'=\Op L^{-1}\big(L_2(\R^d)\big)$ are both Hilbert spaces, while the underlying Riesz map
is \revise{$\Op J_\Spc H=\Lop^{-1}\Op L^{-1\ast} 
: \Spc H \to \Spc H'$}.
\end{enumerate}
When we choose $\Spc X=L_2(\R^d)$, the space $\Spc X$  is its own dual and is usually associated with Tikhonov
regularization. The other fundamental scenario  is 
 $(\Spc X,\Spc X')=
\big(C_0(\R^d),$ $\Spc M(\R^d)\big)$. There, $C_0(\R^d)$
is the space of continuous functions that decay at infinity, while its
dual $\Spc M(\R^d)=\big(C_0(\R^d)\big)'$ is the space of Radon measures on $\R^d$ \cite{Rudin1991}.
The regularization norm $\|\cdot\|_\Spc M$, also known as ``total variation'' in the sense of measure theory, is often used in applications since it promotes sparsity \cite{Candes2013b,Denoyelle2017,Duval2015}.
\begin{corollary}[Representer theorem for mixed-norm regularization]
\label{Theo:L2M}
Let us consider the following setting:
\begin{itemize}
\item An operator $\Op L_1^\ast$ that is injective on $C_0(\R^d)$ and the corresponding 
dual pair of Banach spaces $\Spc B_1=\Op L_1^\ast\big(\Spc C_0(\R^d)\big)$ and $\Spc B'_1=\Op L_1^{-1}\big(\Spc M(\R^d)\big)$.
\item An operator $\Lop_2^\ast$ that is injective on $L_2(\R^d)$ and the corresponding 
dual pair of Hilbert spaces $\Spc H_2=\Op L_2^\ast\big(L_2(\R^d)\big)$ and $\Spc H'_2=\Op L_2^{-1}\big(L_2(\R^d)\big)$.
\item The linear measurement operator 
$\V \nu=(\nu_1,\dots,\nu_M): \Spc B'_1 + \Spc H'_2 \to \R^M$ with $\nu_1,\dots,\nu_M \in \Spc B_1 \cap \Spc H_2$.

\item A strictly convex loss functional $E: \R^M \times \R^M \to \R^{+}$.
\item Two adjustable weights $\lambda_1,\lambda_2 \in \R^+$.
\end{itemize}
Then, for any given $\V y\inR^M$, the two-component regularized inverse problem
\begin{align}
\label{Eq:MixedOptimizationProb}
S=\arg \min_{f =f_1 + f_2: (f_{1},f_{2})\in \Spc B'_1 \times \Spc H'_2}  \left(
E\big(\V y, \V \nu(f)\big)+  { \lambda_1\|\Lop_1f_1\|^2_\Spc M +  \lambda_2\|\Lop_2f_2\|^2_{L_2} }\right)
\end{align}
has a nonempty (and weak$^\ast$-compact) solution set $S$.
Any solution can be written as $f_0=f_{0,1}+f_{0,2}$ 
with $f_{0,1}\in S|_{\Spc B_1}$ and a unique
\begin{align}
f_{0,2}&=\sum_{m=1}^M \widetilde{a}_m \varphi_m,\label{Eq:Solmixed2}
\end{align}
with $\varphi_m=\revise{\Op L_2^{-1}\Op L_2^{-1\ast}
\{\nu_m\}}$ and $(\widetilde{a}_m)\inR^N$, that is common to all solutions.
Moreover, the extremal points of $S|_{\Spc B_1}$ can all be expressed as
\begin{align}
\label{Eq:Solmixed1}
f_{0,1}&=\sum_{k=1}^{K_1} c_k h_1(\cdot,\V \tau_k)
\end{align}
with $K_1\le M$, $(c_k)\inR^{K_1}$ and $h_1(\cdot,\V \tau_k)=\Op L_1^{-1}\{\delta(\cdot-\V \tau_k)\}$, where $\V \tau_1, \dots, \V \tau_{K_1} \inR^d$ are adaptive 
centers.
\end{corollary}
\begin{proof} Problem \eqref{Eq:MixedOptimizationProb} is a special case of
\eqref{Eq:MultiOptimizationProb} with $N=2$, 
$\Spc X_1'=\Spc B'_1$, $\Spc X_2'=\Spc H'_2$, 
$\psi(\cdot)=|\cdot|^2$ (strictly convex), and $\|(x_1,x_2)\|_{\Spc Z'}=\sqrt{\lambda_1 x_1^2 + \lambda_2 x_2^2}$.
The hypotheses are such that the condition for the validity of
Theorem \ref{Theo:Directprod} are met. 
The parametric form of the second (unique) solution component is then given by \eqref{Eq:f0MultiHilbert} with $\widetilde{a}_m=\alpha_2 a_m$, $\nu_{m,n}=\nu_m$ 
and \revise{$\Op J_{\Spc H_2}=
\Op L_2^{-1}\Op L_2^{-1\ast}$}.
To characterize the extremal points of the first component, we use \eqref{Eq:f0MultiSparse}.
To identify the relevant atoms, we recall that the extremal points of the unit ball in $\Spc M(\R^d)$ are given by $\{\pm\delta(\cdot-\V \tau)\}_{\V \tau \in \R^d}$ \cite{Unser2020b}. 
We then make use of the fourth property in Proposition \ref{Prop:IsometricIso} to obtain the form of an extremal
point of the unit ball in $\Spc B_1'$: $e_k=\pm\Op L_1^{-1}\{\delta(\cdot-\V \tau_k)\}$ with shift parameter $\V \tau_k\inR^d$.
\end{proof}

\revise{As slight variant of \eqref{Eq:MixedOptimizationProb}, we may consider
an outer $\ell_1$ norm with $\psi(t)=t$,
which yields a regularization of the form $\lambda_1\|\Op L_1f_1\|_{\Spc M}
+ \lambda_2\|\Op L_2f_2\|_{L_2}$. The resulting solution takes the same functional form, with the caveat that there may no longer exist a single Hilbert-space component $f_{0,2}$ that would be common to all solutions.}

\section{Minimization of Semi-Norms} 
\label{Sec:SemiNorm}
We now consider the scenario of a native Banach space $\Spc X'$ that has a direct-sum decomposition as
$\Spc X'=\Spc U' \oplus \Spc N_\V p$, where $ \Spc U'$ is the dual of some primary Banach space $(\Spc U, \|\cdot\|_{\Spc U})$ and where the complementary space 
$\Spc N_\V p$ is spanned by the finite-dimensional basis $\V p=(p_1,\dots,p_{N_0})$.
Since $\Spc N_\V p={\rm span}\{p_n\}_{n=1}^{N_0}$ is of dimension $N_0$ and hence also reflexive,
the same holds true for its continuous dual $\Spc N'_\V p$.
Moreover, due to the direct-sum property, there exists a unique biorthonormal set of generators $p^\ast_1,\dots,p^\ast_{N_0} \in \Spc X$ such that
$\Spc N'_\V p={\rm span}\{p^\ast_n\}_{n=1}^{N_0}=\Spc N_{\V p^\ast}$ and
\begin{align*}
\langle p^\ast_m,p_n\rangle&=\delta_{m,n}\\
\langle p^\ast_n,s \rangle&=0
\end{align*}
for any $s \in \Spc U'$ and $m,n\in \{1,\dots,N_0\}$. This allows us to specify the 
canonical projector $\Proj_{\Spc N_\V p}: \Spc X' \to \Spc N_\V p$ as
\begin{align}
\label{Eq:ProjNp}
\Proj_{\Spc N_\V p}\{f\}=\sum_{n=1}^{N_0}  \langle p^\ast_n, f\rangle p_n
\end{align}
for any $f \in \Spc X'$. This identification also yields the complementary projector $\Proj_{\Spc U'}: \Spc X' \to \Spc U'$ as
$\Proj_{\Spc U'}=(\Identity-\Proj_{\Spc N_\V p})$.
\revise{Likewise, by interchanging the role of the synthesis and analysis functionals in \eqref{Eq:ProjNp}, we identify
the canonical projector $\Proj_{\Spc N_{\V p^\ast} }: \Spc X \to \Spc N_{\V p^\ast}$ as
$$
\Proj_{\Spc N_{\V p ^\ast}}\{\nu\}=\sum_{n=1}^{N_0}  \langle p_n, \nu\rangle p^\ast_n
$$
for any $\nu \in \Spc X$.} We now have the means to specify and bound the norm of any $f \in \Spc X'$ as
$$
\|f\|_{\Spc X'}=(\|f\|_{\Spc U'}, \|\Proj_{\Spc N_\V p}\{f\}\|_{\Spc N_p})_{\Spc Z'} \le \|f\|_{\Spc U'} + \|\Proj_{\Spc N_\V p}\{f\}\|_{\Spc N_p}
$$
where the functional $f \mapsto \|f\|_{\Spc U'}\eqdef \|\Proj_{\Spc U'}f\|_{\Spc U'}$ is a semi-norm (resp., a norm) over
$\Spc X'$ (resp., $\Spc U'$).
Because the $p_n$ are linearly independent, we also note that the standard description of
$\Spc U'$ as the complement of $\Spc N_\V p$ in $\Spc X'$, given by
$$\Spc U'=\{s \in \Spc X': \Proj_{\Spc N_\V p}\{s\}=\sum_{n=1}^{N_0} \langle p^\ast_n,s \rangle p_n=0\},$$
is equivalent to 
$$
\Spc U'=\{s \in \Spc X': {\V p}^\ast(s)=\V 0\}
$$
where ${\V p}^\ast(f)=(\langle p^\ast_1,f \rangle,\dots,\langle p^\ast_{N_0},f \rangle)$ is a vector-valued functional $\Spc X'\to \R^{N_0}$. Likewise, we have that
$$
\Spc U=\{u \in \Spc X: \V p(u)=\V 0\},
$$
which, once again, capitalizes on the biorthonormality of $({\V p}^\ast,\V p)$. 

The consideration of such a direct-sum decomposition of a Banach space $\Spc X'$ is relevant to inverse problems because it suggests that one can substitute the original regularization term $\|f\|_{\Spc X'}$ by the semi-norm $\|f\|_{\Spc U'}$ when one wants to favor solutions with a strong contribution in $\Spc N_\V p$, the null space of the semi-norm. 
This is a standard technique in spline theory, albeit within the classical context of RKHS spaces
with $\|f\|_{\Spc U'}=\|\Lop f\|_{L_2}$, where $\Lop$ is a suitable differential operator (e.g., a higher-order derivative or fractional Laplacian) with a null space $\Spc N_\V p$ that consists of polynomials of degree $n$ \cite{deBoor1966, Duchon1977, Bezhaev2001}. We now show how this technique can be extended in full generality to Banach spaces.
The basic requirement for this extension is that the inverse problem be well-posed over $\Spc N_\V p$.
This is made explicit in \eqref{Eq:Wellposed}, which is equivalent to the fourth condition in Theorem \ref{Theo:GeneralRepSemiBanach}.

\begin{theorem}[General representer theorem for Banach semi-norms]
\label{Theo:GeneralRepSemiBanach}
Let us consider the following setting:
\begin{itemize}
\item A dual pair of Banach spaces $(\Spc X=\Spc U \oplus \Spc N_{\V p^\ast},\Spc X'=\Spc U' \oplus \Spc N_\V p)$, where $\Spc N_\V p=\Spc N'_{\V p^\ast}$ is the vector space spanned by the finite-dimensional basis $\V p=(p_1,\dots,p_{N_0})$.
\item The analysis subspace $\Spc N_\V \nu={\rm span}\{\nu_m\}_{m=1}^M \subset \Spc X$,
with the $\nu_m$ being linearly independent and $M>N_0$.
\item The linear measurement operator 
$\V \nu: \Spc X' \to \R^M: f \mapsto \big(\langle \nu_1,f \rangle, \dots,\langle \nu_M, f \rangle\big)$.

\item \revise{The vectors $\M v_1,\dots,\M v_{N_0} \in\R^M$ with 
$[\M v_n]_m=\langle \nu_m,p_n \rangle$ are linearly independent;
they admit a complementary set $\{\M u_1, \dots, \M u_{M-N_0}\}$ in $\R^M$ such that $\R^M={\rm span} \{\M v_n\}_{n=1}^{N_0}  \oplus {\rm span} \{\M u_m\}_{m=1}^{M-N_0}$.}

\item A  proper, lower-semicontinuous, coercive, 
and convex loss functional $E: \R^M \times \R^M \to \R^{+}\cup \{+\infty\}$.
\item Some arbitrary strictly increasing and convex function $\psi: \R^+ \to \R^+$. 
\end{itemize}
Then, for any fixed $\V y\inR^M$, the solution set of the generic optimization problem
\begin{align}
\label{Eq:GenericOptimizationProb2}
S=\arg \min_{f \in \Spc X'} \left( E\big(\V y, \V \nu(f)\big)+ \psi\left( \|f\|_{\Spc U'}\right)\right)
\end{align}
is nonempty, convex, and weak$^\ast$-compact. 

If $E$ is strictly convex, or  if it imposes the  equality constraint
$\V y = \V \nu (f)$, then any solution $f_0 \in
S\subset \Spc X'$ has a unique decomposition as $f_0=p_0+s_0$ with
$\revise{p_0} \in \Spc N_\V p$
and $s_0 \in \Spc U'$ the $(\Spc U',\Spc U)$-conjugate of a common
$\widetilde{\nu}_0 \in \Spc U$
whose generic form is
\begin{align}
\label{Eq:dualmodel}
\widetilde{\nu}_0=\sum_{m=1}^{M-N_0} a_m \widetilde{\nu}_m \in \Spc N_\V \nu \cap \Spc U
\end{align}
\revise{ 
with suitable coefficients $\V a=(a_m) \inR^{M-N_0}$ and 
reduced basis functions $\widetilde{\nu}_m=\widetilde{\M u}_m^T\V \nu \in \Spc U$, where $\widetilde{\M u}_m \in\R^M$ is the unique (biorthogonal) vector such that $\widetilde{\M u}^T_m\M v_n=0$ and $\widetilde{\M u}^T_m\M u_{m'}=\delta_{m,m'}$
for any $m,m' \in \{1,\dots,M-N_0\}$ and $n\in \{1,\dots,N_0\}$.}
%
%
%
%
%

Depending on the Banach characteristics of $\Spc U'$, this then results in the following explicit description of the solution(s):
\begin{itemize}
\item If $\ \Spc U'$ is a Hilbert space and $\psi$ is strictly convex, then
the solution $f_0$ is unique and admits the linear representation 
\begin{align}
\label{Eq:f0lin2}
f_0=\sum_{m=1}^{M-N_0} a_m \varphi_m +\sum_{n=1}^{N_0} b_{n} p_n,
\end{align}
with coefficients $(\V a,\V b)\inR^{M}$ and 
basis functions $p_n\in \Spc N_\V p$, $\varphi_m=\Op J_\Spc U\{\widetilde{\nu}_m\}\in \Spc U'$, where
$\Op J_\Spc U$ is the Riesz map $\Spc U \to \Spc U'$.
\item If $\ \Spc U'$ is a strictly convex Banach space and $\psi$ is strictly increasing, then
the solution is unique and admits the parametric representation
\begin{align}
\label{Eq:f0nonlin2}
f_0=\Op J_{\Spc U}\left\{\sum_{m=1}^{M-N_0} a_m \widetilde\nu_m\right\} +\sum_{n=1}^{N_0} b_{n} p_n
\end{align}
where
$\Op J_\Spc U$ is the (nonlinear) duality operator $\Spc U \to \Spc U'$.
\item Otherwise, when $\Spc U'$ is not strictly convex,  the solution set is the weak*-closure of the convex hull of its extremal points, which can all be expressed as
\begin{align}
\label{Eq:f0extreme2}
f_0=\sum_{k=1}^{K_0} c_k e_k +\sum_{n=1}^{N_0} b_{n} p_n
\end{align}
for some $K_0\le (M-N_0)$, $c_1,\dots,c_{K_0} \inR$, where $e_1,\dots,e_{K_0}\in \Spc U'$ are some extremal points of the unit ball $B_{\Spc U'}=\{s \in \Spc U: \|s\|_{\Spc U'}\le 1\}$.
The vector $\V b=(b_{n})\inR^{N_0}$ that characterizes the null-space component of $f_0$ is unique and common to all solutions whenever
$E$ is strictly convex and $\Spc N_{\V p^\ast}\subset \Spc N_\V \nu$.
\end{itemize}

\end{theorem}


Before proceeding with the proof of Theorem \ref{Theo:GeneralRepSemiBanach}, we detail the way in which the reduced basis  $\widetilde{\V \nu}=(\widetilde\nu_1, \dots,\widetilde\nu_{M-N_0})$ in \eqref{Eq:dualmodel} 
 is constructed. To that end, we first define the cross-correlation matrix $
 \M V=[{\M v}_1\  \cdots \ {\M v}_{N_0}]=\V \nu(\V p) \inR^{M \times N_0}$ with
$[\M V]_{m,n}=\langle \nu_m,p_n\rangle$.

\begin{proposition}[Direct-sum decomposition of the measurement space]
\label{Prop:DirectSum}
Let $\V \nu=(\nu_1,\dots,\nu_M) \in \Spc X^M$ with $\Spc X=\Spc U \oplus \Spc N_{\V p^\ast}$ and $\V p=(p_1,\dots,p_{N_0}) \in (\Spc X')^{N_0}$
be two vectors of linear functionals such that the matrix $\M V=\V \nu(\V p) \inR^{M \times N_0}$ is of rank $N_0$.
Then, one can always find three matrices $\M U\inR^{M \times (M-N_0)}$, $\widetilde{\M V}\in \R^{M \times N_0}$, and $\widetilde{ \M U} \inR^{M \times (M-N_0)}$ 
such that 
\begin{align}
\label{Eq:Biortho2}
\left[
\begin{array}{c}
 \widetilde{\M U}^T \\
    \widetilde{\M V}^T
 \end{array}
\right] \left[
\begin{array}{cc}
 {\M U} &     {\M V} 
 \end{array}
\right]
=\M I_M.
\end{align}
Based on these matrices, $\V \nu \in \Spc X^M$ has a unique and reversible decomposition as
\begin{align}
\label{Eq:Decompnu}
\V \nu 
= {\M U} \widetilde{\V \nu} + {\M V} \widetilde{\V p}^\ast,
\end{align}
where
\begin{align}
\label{Eq:ReducedBasis}
\widetilde{\V \nu}&=(\widetilde{\nu}_1,\dots, \widetilde{\nu}_{M-N_0})=\widetilde{\M U}^T\V \nu  \in \Spc U^{M-N_0}\\
\label{Eq:ReducedBasis2}\widetilde{\V p}^\ast&=(\widetilde{p}^\ast_1,\dots, \widetilde{p}^\ast_{N_0})=\widetilde{\M V}^T\V \nu
\in \Spc X^{N_0}.
\end{align}
In effect, this yields a decomposition of the measurement space $\Spc N_{\V \nu}={\rm span}\{{\nu}_m\}_{m=1}^{M}$ as $\Spc N_{\V \nu}=\Spc N_{\widetilde{\V p}^\ast} \oplus \Spc N_{\widetilde{\V \nu}}$ 
with $\Spc N_{\widetilde{\V \nu}}={\rm span}\{\widetilde{\nu}_m\}_{m=1}^{M-N_0} \subset \Spc U$.
In particular, if $\Spc N_{\V p^\ast}\subset \Spc N_\V \nu$, then there is a unique matrix 
$\widetilde{\M V}\in \R^{M \times N_0}$ of rank $N_0$ such that
$\widetilde{\M V}^T  \V \nu={\V p^\ast}$ 
and such that the decomposition still applies with $\Spc N_{\V \nu}=\Spc N_{\V p^\ast} \oplus \Spc N_{\widetilde{\V \nu}}$.
\end{proposition}

\begin{proof} 
Since the 
vectors $\M v_1,\dots,\M v_{N_0} \in \R^M$ are linearly independent,
they can always be completed by adding some vectors
$\M v_{N_0+1}=\M u_1,\dots,\M v_{M}=\M u_{M-N_0}$ to form a basis of $\R^M$.
The linear independence of the resulting family (basis property) is equivalent to the existence of a unique dual basis
$\widetilde{\M v}_1,\dots,\widetilde{\M v}_{M} \in \R^M$
such that
\begin{align}
\label{Eq:biortho}
\langle \widetilde{\M v}_m, \M v_n\rangle =\widetilde{\M v}^T_m \M v_n=\delta_{m,n}
\end{align}
for $m,n \in \{1,\dots,M\}$ (biorthonormality property). 
This means that any vector \revise{$\V y\inR^M$} has a unique (and reversible) decomposition as $\V y=\sum_{m=1}^M \langle  \widetilde{\M v}_m, \V y\rangle \M v_m$.
By collecting the expansion coefficients in the two vectors
$\widetilde {\V b}=(\langle  \widetilde{\M v}_1, \V y\rangle,\dots,\langle  \widetilde{\M v}_{N_0},\V y\rangle)$
and $\widetilde {\V y}=(\langle  \widetilde{\M v}_{N_0+1}, \V y\rangle,\dots,\langle  \widetilde{\M v}_M,\V y\rangle)$ and identifying the matrices
$\widetilde{\M V}=[\widetilde{\M v}_{1}\ \cdots\ \widetilde{\M v}_{N_0}]$,
$\M U=[\M v_{N_0+1}\ \cdots \ \M v_{M}]$, and $\widetilde{\M U}=$ ${[\widetilde{\M v}_{N_0+1}\ \cdots\ \widetilde{\M v}_{M}]}$, we then observe that 
the decomposability of $\V y\inR^M$ is equivalent to
\begin{align}
\label{Eq:directsumRM}
\V y 
=
 \left[
\begin{array}{cc}
 {\M U}  &  {\M V} 
 \end{array}
\right] \left[
\begin{array}{c}
\widetilde{\V y}  \\
\widetilde{\V b}  
 \end{array}
\right] 
\end{align}
with
\begin{align}
\label{Eq:uvector2}
\widetilde{\V y} 
=\widetilde{\M U}^T \V y \in \R^{M-N_0}, \quad
\widetilde{\V b} 
=  \widetilde{\M V}^T \V y \in \R^{N_0}.
\end{align}
Likewise, the enabling biorthonormality property \eqref{Eq:biortho} is equivalent 
to the invertibility condition \eqref{Eq:Biortho2}.

By substituting $\V y\inR^M$, $\widetilde{\V y}\inR^{M-N_0}$, and $\widetilde{\V b}\inR^{N_0}$ by $\V \nu(f), \widetilde{\V \nu}(f)$, and
$\widetilde{\V p}^\ast(f)$, respectively, we
 then rephrase \eqref{Eq:directsumRM} and \eqref{Eq:uvector2} in term of functionals, which yields the reversible decomposition described by
 \eqref{Eq:Decompnu}, \eqref{Eq:ReducedBasis}, and \eqref{Eq:ReducedBasis2}.
To prove that $\Spc N_{\widetilde{\V \nu}}\subset \Spc U=\{u \in \Spc X: \V p(u)=0\}$, we invoke the invertibility condition \eqref{Eq:Biortho2}, which yields $\V p(\widetilde{\V \nu})=
\big(\widetilde{\V \nu}(\V p)\big)^T=(\widetilde{\M U}^T\M V)^T=\M 0^T$.
Since, for any $\V a\inR^M$, we have that
$\V a^T\M U\widetilde{\V \nu}\in {\rm span}\{ \widetilde{\nu}_n\}_{n=1}^{M-N_0}$ and $\V a^T\M V\widetilde{\V p}^\ast \in {\rm span}\{\widetilde{p}^\ast_n\}_{n=1}^{N_0}$ with the linear expansion of $\V a^T\V \nu \in \Spc N_\V \nu$ in the corresponding basis being reversible, we can interpret \eqref{Eq:Decompnu} as the direct-sum decomposition
 $\Spc N_\V \nu=\Spc N_{\widetilde{\V \nu}}\oplus \Spc N_{\widetilde{\V p}^\ast}$.

The inclusion $\Spc N_{\V p^\ast}\subset \Spc N_\V \nu$, together with the linear independence of the $\nu_m$, is equivalent to the existence of a unique transformation matrix $\widetilde{\M V}^T$
of rank $N_0$ such that  ${\V p}^\ast=\widetilde{\M V}^T \V \nu$. 
While this sets the matrix $\widetilde{\M V} \inR^{M \times N_0}$, one is still left with sufficiently many degrees of freedom
to select the complementary matrices $\M U$ and $\widetilde{\M U}$ such that
\eqref{Eq:Biortho2} holds.
\end{proof}

We note that, irrespective of whether we fix $\widetilde{\M V}$ (second part of Proposition \ref{Prop:DirectSum}) or not,  there are generally infinitely many admissible choices
for $\M U$ 
in \eqref{Eq:Biortho2} 
and, hence, for the construction of the reduced basis $\widetilde{\V \nu}$ defined by \eqref{Eq:ReducedBasis}.
\revise{This does not contradict the unicity of \eqref{Eq:dualmodel}. Indeed, different choices of extension correspond to different biorthogonal bases $\widetilde{\V \nu}=(\widetilde\nu_1, \dots,\widetilde\nu_{M-N_0})$ of the same subspace.}


\subsection{Proof of Theorem \ref{Theo:GeneralRepSemiBanach}}

\begin{proof} 
\item {\em (i) Existence}: The classical conditions that ensure the existence of a minimizer of the functional $J(f)=E\big(\V y, \V \nu(f)\big)+ \psi\left( \|f\|_{\Spc U'}\right)$ are that $J(f)$ should be proper, convex, (weak$\ast$-)lower-semi-continuous, and coercive over $\Spc X'$.
These higher-level properties also imply that the solution set $S$ is convex, and weak$^\ast$-compact.

The first three conditions follow from the listed assumptions and the general properties of a (semi-)norm---see argumentation in the proof of Theorem \ref{Theo:GeneralRepBanach} in \cite{Unser2019c}.
To establish coercivity, we recall that the hypothesis $\nu_m \in \Spc X$ implies the continuity of $\nu_m: \Spc X' \to \R$ due to the continuous embedding of $\Spc X$ in its bidual $\Spc X''=(\Spc X')'$. Consequently, there exists some constant $A>0$ such that
\begin{align}
\|\V \nu(f)\|_2 \le  A \|f\|_{\Spc X'}
\end{align}
for all $f \in \Spc X'$. Likewise, the linear independence of the $\M v_n$ 
and the property that all finite-dimensional norms are equivalent implies the existence of $B>0$ such that, for any $p\in \Spc N_\V p$,
\begin{align}
\label{Eq:Wellposed}
B \|p\|_{\Spc N_\V p}\le \|\V \nu(p)\|_2.
\end{align}
By using the direct-sum decomposition
$f=s+p$ with $(s,p) \in \Spc U' \times \Spc N_\V p$, $\|f\|_{\Spc X'}\le \|s\|_{\Spc U'}+\|p\|_{\Spc N_{\V p}}$, and $\|s\|_{\Spc U'}=\|f\|_{\Spc U'}$, we readily deduce that
\begin{align}
\|\V \nu(f)\|_2 & \ge \|\V \nu(p)\|_2 - \|\V \nu(s)\|_2 
\ge B \|p\|_{\Spc N_\V p} -A \|s \|_{\Spc U'} \ge  B \|f\|_{\Spc X'} -(A + B) \|f \|_{\Spc U'},
\label{Eq:bound2} 
\end{align}
where we have made use of  the triangle inequality and the two previous bounds.
Let us now consider some sequence $(f_m)$ in $\Spc X'$ with $f_m=(s_m,q_m) \in \Spc U' \times \Spc N_\V p$ such that $\|f_m\|_{\Spc X'} \ge \|f_n\|_{\Spc X'}$ for $m\ge n$ and $\lim_{m\to\infty} \|f_m\|_{\Spc X'}=\infty$.
Then, there are two possible asymptotic behaviors for the norm of $s_m=\Proj_{\Spc U'} f_m$:
\begin{enumerate}
\item The quantity $\|s_m\|_{\Spc X'}=\|f_m\|_{\Spc U'}\to \infty$ as $m\to \infty$, in which case 
$J(f_m)\to \infty$ due to the unboundedness of $\psi: \R^+ \to \R^+$ at infinity.
\item There exists a constant $C$ such that $\|f_m\|_{\Spc U'}\le C$ for all $m$. 
By invoking \eqref{Eq:bound2}, we get that $\|\V \nu(f_m)\|_2\to \infty$ as $m\to\infty$, which, in turn, gives $J(f_m)\to\infty$, due to the
coercivity of $E(\cdot,\V y)$.
\end{enumerate}
In summary,
$J(f)\to \infty$ as $\|f\|_{\Spc X'}\to \infty$, which is the required coercivity property.\\


\item {\em (ii) Representation of a solution}:  
The underlying direct-sum property implies that  any $f_0 \in S \subset \Spc X'$
has a unique decomposition as $f_0=s_0+p_0$ with $(s_0,p_0) \in \Spc U' \times \Spc N_\V p$.
To derive the parametric form of a solution, we momentarily assume that $p_0$ (and, hence, $\V \nu(p_0)\inR^M$)
and $\V y_0=\V \nu(f_0) \inR^M$ are known. By making use of the decomposition of the measurement space in Proposition \ref{Prop:DirectSum}, we observe that the penalized component $s_0 \in \Spc U'$ solves the equivalent constrained-optimization problem
\begin{align}
\label{Eq:s01}
s_0 \in S_{p_0, \V y_0}=\arg \min_{s \in \Spc U'} \|s\|_{\Spc U'} \mbox{ s.t. }
\V y_0-\V \nu(p_0)
=\V \nu(s)=\M U\widetilde{\V \nu}(s) + \M V \widetilde{\V p}^\ast(s),\end{align}
where $\widetilde {\V \nu}=\widetilde{\M U}^T \V \nu 
\in \Spc U^{M-N_0}$ and $\widetilde {\V p}^\ast=\widetilde{\M V}^T \V \nu \in \Spc X^{N_0}$.
We now show that the effective number of linear constraints in 
\eqref{Eq:s01} is actually $(M-N_0)$ and not $M$, as may be thought on first inspection.
To that end, we multiply the 
linear constraint by $\widetilde{\M U}^T\in \R^{(M-N_0) \times M}$ on both sides and use the properties that $\widetilde{\M U}^T\M U=\M I_{M-N_0}$ and
$\widetilde{\M U}^T\M V=\M 0$ (see \eqref{Eq:Biortho2} in Proposition \ref{Prop:DirectSum}). This yields
\begin{align}
\label{Eq:s02}
s_0 \in S_{p_0, \V y_0}=\arg \min_{s \in \Spc U'} \|s\|_{\Spc U'} \mbox{ s.t. }
\widetilde{\V \nu}(s)=\widetilde{\V y}_0,
\end{align}
where $\widetilde{ \V y}_0=\widetilde{\M U}^T \big(\V y_0 -\V \nu(p_0)\big)=\widetilde{\M U}^T \V y_0 \in \R^{M-N_0}$. This latter simplification occurs because
$\widetilde{\M U}^T\V \nu(p) =\widetilde{ \V \nu}(p)=0$ for all $p \in \Spc N_\V p$ by construction. 
The theoretical significance of the cancellation of  $\widetilde{\M U}^T\V \nu(p_0)$ is that the above manipulation does not depend on $p_0$, so that $S_{p_0, \V y_0}=S_{\V y_0}$. In effect, this
means that the characterization of the optimal $s_0$ in \eqref{Eq:s02} 
holds for all solutions that share the same measurements $\V y_0$.
This is true, in particular, when $E$ is strictly convex, by a standard argument in convex analysis.
The description of
$s_0 \in \Spc U'$ as the $(\Spc U',\Spc U)$-conjugate of a common
$\widetilde{\nu}_0 \in \Spc N_{\widetilde{\V \nu}}$, as well as  \eqref{Eq:f0lin2},
\eqref{Eq:f0nonlin2},
and \eqref{Eq:f0extreme2}, then 
follow from Theorem \ref{Theo:GeneralRepBanach}.

For the special scenario $\Spc N_{\V p^\ast} \subset \Spc N_\V \nu$, 
we select $\widetilde{\M V}$ such that ${\V p^\ast}=\widetilde{\M V}^T\V \nu$ (see the second part of Proposition \ref{Prop:DirectSum}) and are then able to
obtain the expansion of coefficients of $p_0=\Proj_{\Spc N_\V p}\{f_0\}$ directly from the measurements
as $\V b={\V p^\ast}(f_0)=\widetilde{\M V}^T\V y_0$, which establishes this part of the solution as well
for all $f_0\in S$.
\end{proof}

When the (semi)-norm $\|\cdot\|_{\Spc U'}$ is strictly convex, Theorem \ref{Theo:GeneralRepSemiBanach} states that the unique solution $f_0$ lives in a finite-dimensional manifold that is parameterized by $\V b \in \R^{N_0}$ (for the null-space component $p_0$)
and $\V a \in \R^{M-N_0}$ (for the preimage $\widetilde{\nu}_0$ of the penalized component
$s_0=(\widetilde{\nu}_0)^\ast$).
While the two primary expansions are linear, the high-level ingredient of the representation is the duality mapping, 
which introduces a nonlinearity in the non-Hilbert scenario.
At any rate, the main point is that the intrinsic dimensionality of the solution space is still $M$, as in the case of Theorem 1, except that the repartition is now very different, with the contribution of the null-space component $p_0\in \Spc N_\V p$ being maximized since it is no longer penalized.

An important outcome of Theorem \ref{Theo:GeneralRepSemiBanach} is that it becomes possible to characterize the full solution set $S$ via the specification of a single pair $(p_0,\widetilde \nu_0) \in \Spc N_\V p \times \Spc U$.
For the challenging cases where there are multiple solutions (third scenario), this requires the additional assumption
that $\Spc N_{\V p^\ast} \subset \Spc N_\V \nu$, which has the desirable effect of decoupling the determination $p_0$ from that of $\widetilde\nu_0$. It turns out that this decoupling 
is applicable to most practical problems that involve a semi-norm regularization. The key is that it is generally possible to adapt the semi-norm topology to the problem at hand by selecting a biorthonormal system $({\V p}^\ast,\V p)$ with
$p^\ast_1,\dots,p^\ast_{N_0} \in {\rm span}\{\nu_m\}_{m=1}^M \subset \Spc X$ (see, for instance, \cite{Unser2017, Unser2019a}).

\subsection{Connection with Prior Works}

The result for Hilbert spaces in Theorem \ref{Theo:GeneralRepSemiBanach}  is well known, although it is rarely described in this form where the parame\-terization is tight.
The result for Banach spaces is new, to the best of our knowledge.
The result for the non-strictly convex case overlaps that of two recent papers by Bredies-Carioni \cite{Bredies2018} and Boyer {\em et al.} \cite{Boyer2018}, in which these authors independently establish the existence of solutions of the form \eqref{Eq:f0extreme2} in similar scenarios.
The novel element here is the statement about the form of all extremal points as well as the identification of the configurations where $p_0$ is unique.
Finally, we are not aware of any prior work (except \cite{Unser2019c}) where these various scenarios have been unified.

For the particular scenario where $\Op T=\Lop^{-1}$ is an isomorphism from
$\Spc M(\R^d)$ (the space of Radon measures on $\R^d$) to $\Spc U'=\Op T\big(\Spc M(\R^d)\big)$
and $\Spc X'=\Spc U' \oplus \Spc N_\V p$ with 
$\|f\|_{\Spc U'}=\|\Lop f\|_{\Spc M}$, we recover the results of
\cite{Unser2017} (when the operator $\Lop$ is spline-admissible) and \cite[Theorem 1]{Flinth2018} 
by observing that the generic form of the extremal points of the unit regularization ball are
$e_k=\Op T\{\delta(\cdot-\V x_k)\}$ with $\V x_k \inR^d$. 
The key there is that the extremal points of the unit ball in $\Spc M(\R^d)$ are
the signed shifted Dirac measures $\{\pm\delta(\cdot-\V x)\}_{\V x \inR}$,
which are then isometrically mapped into $\Spc U'$ in accordance with Proposition \ref{Prop:IsometricIso}.
To make this more concrete, we recall that an operator $\Lop: \Spc X' \to \Spc M(\R^d)$ is called spline-admissible if 
\begin{enumerate}
\item it is linear shift-invariant;
\item it has a finite-dimensional null space
 $\Spc N_\Lop=\{f\in \Spc X':\Lop f=0\}={\rm span}\{p_n\}_{n=1}^{N_0}$;
 \item it admits a Green's function $\rho_\Lop: \R^d \to \R$ (of slow growth) such that $\rho_\Lop=\Lop^{-1}\{\delta\}$.
 \end{enumerate}
In that scenario, we find\footnote{There is one subtle point in the derivation (see \cite{Unser2017}) because
the correct inverse operator $\Op T=(\Identity-\Proj_{\Spc N_\V p})\Lop^{-1}=\Lop^{-1}_{{\V p^\ast}}$ must incorporate a projection onto $\Spc U'$. This projection depends upon the topology. This yields extremal points of the form $e_k=(\rho_{\Lop}(\cdot-\V x_k)-q_k)$ with $q_k={\Proj_{\Spc N_\V p}\{\rho_{\Lop}(\cdot-\V x_k)\}}\in \Spc N_{\V p}$, which translates into $\tilde b_n$
in \eqref{Eq:spline} being different from $b_{n}$ in \eqref{Eq:f0extreme2}.}
that the extremal points of \eqref{Eq:GenericOptimizationProb2} with $\|f\|_{\Spc U'}=\|\Lop f\|_{\Spc M}$ can be represented as
\begin{align}
\label{Eq:spline}
f_0: \R^d \to \R: \V x \mapsto \sum_{n=1}^{N_0} \tilde{b}_n p_n(\V x) + \sum_{k=1}^{K_0} a_k \rho_{\Lop}(\V x-\V x_k)
\end{align}
with $K_0\le (M-N_0)$, 
which is the generic form of a non-uniform $\Lop$-spline with knots at the $\V x_k\inR^d$ \cite{Schultz1967,Schumaker2007}\cite[Chapter 6]{Unser2014book}.
For instance, for $\Lop=\Dop$ (the derivative operator with $d=1$), $N_0=1$ with $p_1=1$ (the constant function),
while $\rho_{\Dop}(x)$ is the unit step (Heaviside function). It follows that
$f_0$ given by \eqref{Eq:spline} is piecewise-constant with jumps at the $x_k\inR$.

\revise{The setting of Theorem 3 can also be extended to the multicomponent scenarios investigated in Section 4. In essence, one can replace regularization norms by semi-norms, which then adds some corresponding null-space component(s) to the generic form of the solution. Hybrid splines, which can be seen
as the continuous-domain counterpart of the multi-dictionary approach of Section 4.2, provide a powerful example of such composite modeling \cite{Debarre2019Hybrid}. Another useful option is the spline variant of Corollary \ref {Theo:L2M}, which adds a smooth component (with a corresponding $L_2$ regularization) to the solution specified by \eqref{Eq:spline}.
However, the fitting of such augmented models to data is trickier: when the intersection of the null spaces is nontrivial,
it requires the specification of additional boundary conditions to ensure that the decomposition is unique  \cite{Debarre2019Hybrid,Debarre2021continuous}.}

\appendix
\section{Proof of the Last Statement in Theorem \ref{Theo:GeneralRepBanach}}
We already mentioned that this characterization can be derived from Theorem 3.1 of Boyer et al. \cite{Boyer2018}
by viewing an extreme point of the solution set as the degenerate case of a face with dimension $j=0$.
The proof presented in  \cite{Boyer2018} relies on an earlier theorem by Klee \cite{Klee1963}, which is itself based on a foundational result by Dubins on the extreme points of the intersection of a convex set and a series of hyperplanes \cite{Dubins1962}.
Here, we have chosen the latter as our starting point in order to simplify the argumentation.
\begin{theorem}[{Main result of \cite{Dubins1962}}]
\label{Theo:Dubins62}
Consider a topological vector space $V$ over the field of real numbers, a closed and bounded convex set $C\subset V$ and  $M$ hyperplanes $H_1,\ldots,H_M\subset V$. Then, any extreme point of $C\cap  \left( \bigcap_{m=1}^M H_m\right)$ can be written as a convex combination of at most $M+1$ extreme points of $C$.
\end{theorem}
For the Banach space $\Spc X'$, we denote the unit ball of size $\beta$ as $B_{\Spc X',\beta}= \{ f\in \Spc X': \|f\|_{\Spc X'} \leq\beta\}$.  It then directly follows from Theorem \ref{Theo:Dubins62} that any extreme point  of 
$$
\Spc S=   B_{\Spc X',\beta} \cap \{ f\in \Spc X': \V \nu (f)=\V y\} 
$$
can be written as a convex combination of at most $M+1$ extreme points of $ B_{\Spc X',\beta}$.  In what follows, we show that, if 
$$
\beta = \min_{f \in \Spc X'} \|f\|_{\Spc X'} \quad \text{s.t.} \quad \V \nu (f)=\V y,
$$
then any extreme points $f_0$ of $\Spc S$ has the expansion   
\begin{equation}\label{Eq:ExtPointExpansion}
f_0 = \sum_{k=1}^{K} c_k f_k, \quad K\leq M,
\end{equation}
where $f_k \in {\rm Ext}(B_{\Spc X',\beta})$, and $c_k>0$ with $\sum_{k=1}^K c_k =1$. 
The connection with \eqref{Eq:f0extreme} is that \eqref{Eq:ExtPointExpansion} is obviously also expressible in terms of the basis vectors
$e_k=f_k/\beta$ which, due to the homogeneity property of the norm, are extremal points of the unit ball in $\Spc X'$. 

Assume by contradiction that $K=M+1$ and that the set $\{f_1,\ldots,f_{M+1}\}$ is linearly independent.   The set of vectors $\{\V \nu(f_1),\ldots,\V \nu(f_{M+1})\}\subseteq \R^M$ is clearly linearly dependent. Hence, there exists  $(\alpha_m)_{m=1}^{M+1} \neq \V 0$ such that 
\begin{equation}\label{Eq:LinDepend}
\V\nu(\sum_{m=1}^{M+1} \alpha_m \V f_m) = \sum_{m=1}^{M+1} \alpha_m \V\nu(f_m) =\V 0.
\end{equation} 
 Denote $A=\sum_{m=1}^{M+1} \alpha_m$ and consider the function $f_{\epsilon}= f_0 + \epsilon \sum_{m=1}^{M+1} \alpha_m f_m$ for $\epsilon \in \mathbb{R}$. On one hand, for all values of $\epsilon$ with $|\epsilon|< \epsilon_{\max}= \frac{\min_m c_m}{\max_m |\alpha_m| }$, the function 
$$
 \frac{f_{\epsilon}}{1+\epsilon A}= \sum_{m=1}^{M+1} \frac{c_m+\epsilon \alpha_m}{1 + \epsilon A}  f_m 
$$
is in the convex hull of $\{f_1,\ldots,f_{M+1}\}$. Consequently, $\|f_{\epsilon}\|_{\Spc X'} \leq |1+\epsilon|\beta$.  
On the other hand, due to \eqref{Eq:LinDepend},   we have
$$
\V \nu( f_{\epsilon}) = \V \nu(f_0) + \epsilon \sum_{m=1}^{M+1} \alpha_m \V y_m = \V y.
$$
Hence, due to the optimality of $f_0$, we  deduce that 
$$ 
|1+\epsilon A| \beta \geq  \|f_{\epsilon}\|_{\Spc X'} \geq \|f_0\|_{\Spc X'} = \beta, \quad \forall \epsilon \in (-\epsilon_{\max},\epsilon_{\max}).
 $$
This yields that $|1+\epsilon A| \geq 1$ for all $\epsilon \in (-\epsilon_{\max},\epsilon_{\max})$, which implies that $A=0$. Consequently, $f_{\epsilon}\in \Spc S$ for all $\epsilon \in (-\epsilon_{\max},\epsilon_{\max})$. Now, since $f_0$ is an extreme point of $\Spc S$, we deduce from $f_0 = \frac{f_{-\epsilon}+f_{\epsilon}}{2}$ that $f_0=f_\epsilon$ for all  $\epsilon \in (-\epsilon_{\max},\epsilon_{\max})$, and hence, $\sum_{m=1}^{M+1} \alpha_m f_m=0$. This is in contradiction with the linear independence of $\{f_1,\ldots,f_{M+1}\}$. 

\revise{\section*{Acknowledgments}
The authors are thankful to Julien Fageot and Matthieu Simeoni
for helpful discussions and advice. }
\bibliographystyle{elsarticle-num}

%
%
%
%
%
%
\bibliography{DirectSum.bib}
\end{document}